\documentclass[12pt]{elsarticle}
\usepackage[left=1in,right=1in, top=1.2in,bottom=1.2in]{geometry}
\usepackage{EJGTAart}
\usepackage{times}
\usepackage{amssymb,amsthm,latexsym,amsmath,epsfig,pgf}
\usepackage{csquotes}

\usepackage{mathabx}
\usepackage{algorithm}
\usepackage[noend]{algorithmic}
\usepackage{epstopdf}
\usepackage[normalem]{ulem}

\usepackage{graphicx}
\usepackage{tkz-graph}
\usetikzlibrary{calc,arrows.meta,positioning}
\usetikzlibrary{decorations.markings}

\GraphInit[vstyle = Normal] 
\tikzset{
  LabelStyle/.style = {font = \tiny\bfseries },
  VertexStyle/.append style = { inner sep=5pt,
                                font = \tiny\bfseries},
  EdgeStyle/.append style = {->} }
\usetikzlibrary{arrows, shapes, positioning}

\newcommand{\overundercup}[2]{{\underset{#1}{\overset{#2}{\cup}}}}

\volume{{\bf -} (-)}

\firstpage{1}

\runauth{
A modification of two graph-decomposition theorems based on a vertex-removing synchronised graph product
\hspace{2ex} $\arrowvert$\hspace{2ex}
Antoon H. Boode}

\newtheorem{theorem}{Theorem}[section]
\newtheorem*{theorem A}{Theorem A}
\newtheorem*{theorem B}{N\"olker's Theorem}
\newtheorem{lemma}{Lemma}[section]

\theoremstyle{remark}

\theoremstyle{remark}

\newtheorem{claim}{Claim}
\begin{document}

\begin{frontmatter}
\papertitle{\center{\vspace*{-5cm}A modification of \\\vspace*{0.26cm}two graph-decomposition theorems\\\vspace*{0.1cm} based on\\\vspace*{0.26cm} a vertex-removing synchronised graph product}}



\author[label2]{Antoon H. Boode}

\address[label2]{\small {Robotics Research Group, \\InHolland University of Applied Science\\ Alkmaar, the Netherlands}

\vspace*{2.5ex} 
 \normalfont ton.boode@inholland.nl\\
 \normalfont https://orcid.org/0000-0001-5119-6275}
\begin{abstract}
\noindent
Recently, we have introduced two graph-decomposition theorems based on a new graph product (the vertex-removing synchronised product (VRSP)), motivated by applications in the context of synchronising periodic real-time processes. 
In these applications, periodic real-time processes synchronise over actions that have the same label and therefore the same behaviour.
From a process-algebraic point of view, such a synchronising action is executed atomically and at the same time by all processes that have this action in their alphabet.
When these processes are executed on some computer platform, synchronisation leads to context switches of the processes and therefore an increased overhead, which may lead to deadline misses.
But, by combining processes we reduce the number of context switches and therefore reduce the overhead.
We combine these processes by representing the processes by edge-labelled acyclic directed multigraphs, and multiply the graphs by the VRSP.
Next, we transpose the resulting graphs into processes for which there are fewer context switches.
An important aspect of these real-time applications is that they must execute in time. Still, it may happen that the set of processes of the application cannot execute timely and may miss a deadline. Now, by decomposing the graphs and multiplying the graphs by the VRSP in another combination, the processes that are represented by these recombined graphs may execute in time.
The requirements of the recently introduced graph-decomposition theorems are too strict and can be relaxed, whereby more graphs can be decomposed giving more possible combinations for the real-time application.
Therefore, we recall the definition of the VRSP and the two graph-decomposition theorems, we relax the requirements by stating and proving a lemma that decomposes bipartite graphs and use this lemma to state and prove the two (relaxed) graph-decomposition theorems.

\end{abstract}

\begin{keyword}
Vertex Removing Synchronised Graph Product 
\sep Product Graph
\sep Graph Decomposition
\sep Synchronising Processes

Mathematics Subject Classification : 05C76, 05C51, 05C20, 94C15


\end{keyword}

\end{frontmatter}
\section{Introduction}\label{sec:intro}
Recently, we have introduced two graph-decomposition theorems based on a new graph product~\cite{dam}, motivated by applications in the context of synchronising periodic real-time processes, in particular in the field of robotics.
More on the background, definitions, and applications can be found in two conference contributions \cite{boode2014cpa, boode2013cpa}, two journal papers \cite{dam,ejgta} and the thesis of the 
author~\cite{boodethesis}.
In this contribution, we relax some of the requirements of the two graph-decomposition theorems presented in \cite{dam} for which we present a new lemma (Lemma~\ref{lemma1}) that takes bipartite graphs into account.
The lemma is used to relax the requirements of the theorems in~\cite{dam} so we can state and prove the two relaxed decomposition theorems.
Also, we repeat most of the background, definitions, and theorems presented in \cite{dam} here for convenience.
Furthermore, the proofs of Lemma~\ref{lemma1}, Theorem~\ref{theorem_3} and Theorem~\ref{theorem_4} are modelled along the same lines as the proofs of the theorems presented in \cite{dam}.

In \cite{boode2013cpa}, we have modelled periodic real-time processes as directed acyclic labelled multigraphs. These graphs are closely related to state transition systems~\cite{Katoen}. 
The vertices of such a graph represent the states of a periodic real-time process, while the labelled arcs represent actions, i.e., transitions from one state to another. 
The label (in fact, a label pair) on an arc represents the name or type of the action together with the worst-case duration of its execution.
 We give the formal definitions of these graphs in Section~\ref{sec:term}.

Embedded control systems play a crucial role in many application areas. 
In particular, in the field of robotics, it is obvious that these systems (embedded in robots) are key to the functionality and operational behaviour of robots. 
The software of such control systems is usually designed using a general-purpose computing system (not in the robot). These general-purpose computers generally have more processing power and memory available than the embedded control system.
The embedded control system is the target system on which the software will run eventually after it has been designed and validated.
The hardware of the target system is usually much more limited with respect to available memory and processing power.
If the processes that have to run on the target system are periodic and real-time, they have deadlines to fulfil the timing requirements, and they require memory for storing the data and software.

Periodic real-time (robotic) applications can be designed using process algebras like, for example, a calculus of communicating systems~\cite{Milner}, communicating sequential processes~\cite{Hoare}, micro Common Representation Language 2~\cite{Grootte} and finite-state processes~\cite{Magee:2000:CSM:332036}.
During the design phase, the designer distributes the required behaviour over sometimes more than a hundred processes. 
These processes very often synchronise over actions, e.g., to assert whether a subset of the processes will be ready to start executing at the same time.
Due to this synchronisation, such applications usually suffer from a considerable overhead related to so-called  context switches.

In \cite{boode2013cpa}, the vertex-removing synchronised product (VRSP) has been introduced as a means to reduce the number of context switches. This
VRSP is a modification of the well-known Cartesian product of graphs. It is based on the synchronised product due to W\"ohrle and Thomas~\cite{Wohrle04modelchecking}, which is used in model-checking synchronised products of transition systems.

The VRSP reduces the number of context switches and in many cases realises a performance gain for periodic real-time applications.
This is achieved by (repetitively) combining two graphs representing two processes that synchronise over some action.
The combined graph of two graphs then represents a process that will have only one context switch per synchronising action, whereas the two processes separately would each have one context switch per synchronising action \cite{boode2013cpa}.

Using the VRSP, the set of graphs representing a set of different processes can, under certain conditions, be transformed into a new set of graphs. This can be particularly useful if the original set of graphs represents a set of processes that cannot meet their deadline or do not fit into the available memory. 
The aim is that for such a new set of graphs, the processes that they represent meet their deadline and fit into the available memory.  In the worst case, there may be no set of processes with respect to the original set of processes that will do so. In that case, the VRSP cannot result in a suitable solution when applied in any way to the graphs representing the original set of processes. 

One way out, for which we introduce and develop the tools here, is to use the VRSP to enable new combinations of subprocesses of the original set of processes, without changing the functionality and behaviour of the total set of new (sub)processes. 
We accomplish this by decomposing a graph $G$ (representing one of the processes) into two smaller graphs $G_1$ and $G_2$ such that the VRSP of $G_1$ and $G_2$ is isomorphic to $G$. 
It should be noted here, that the graphs $G_1$ and $G_2$ are not subgraphs of $G$, but that they are obtained from $G$ by applying a contraction operation, to be specified later.  

The decomposition of graphs is well known in the literature.
For example, decomposition can be based on the partition of a graph into edge-disjoint subgraphs. 
In our case, in the two graph-decomposition theorems we contract disjoint nonempty subsets of the vertex set $V$ of the edge-labelled acyclic directed multigraph~$G$.
The contraction of a nonempty set $X\subset V$ leads to a graph $G/X$ where all the vertices of $X$ are replaced by one vertex $\tilde{x}$, each arc $uv, u\in V\backslash X, v\in X$ is replaced by an arc $u\tilde{x}$ with $\lambda(u\tilde{x})=\lambda(uv)$, each arc $uv, u\in X, v\in V\backslash X$ is replaced by an arc $\tilde{x}v$ with $\lambda(\tilde{x}v)=\lambda(uv)$, and the arcs with both ends in $X$ are removed.

In the first theorem, we have disjoint nonempty sets $X\subset V$ and $Y=V\backslash X$, giving $G/X$ and $G/Y$.
In the second theorem, we have mutually disjoint nonempty sets $X_1\subset V,X_2\subset V$ and $Y=V\backslash (X_1\cup X_2)$ giving $G/X_1/X_2$ and $G/Y$, where $G/X_1/X_2$ is shorthand for $(G/X_1)/X_2$.
Then, together with additional constraints given in the theorems, we have that $G$ is isomorphic to the VRSP of $G/X$ and $G/Y$ in the first theorem and that $G$ is isomorphic to the VRSP of $G/X_1/X_2$ and $G/Y$ in the second theorem.

In this paper, we recall the definition of the VRSP and the two graph-decomposition theorems given in \cite{dam} and we relax the requirements of these two graph-decomposition theorems. 
For the first theorem, the requirement was that for the arcs that have one end in $X$ and the other end in $Y$ (the set of arcs $[X,Y]$) the label of each arc is distinct.
We relax this requirement in the following manner.
The set of all arcs in $[X,Y]$ with the same label must arc-induce (defined in Section~\ref{sec:term}) a complete bipartite graph. 
For the second theorem, the requirement was that for the arcs that have one end in $X_1$ and the other end in $Y$ (the set of arcs $[X_1,Y]$), the arcs that have one end in $Y$ and the other end in $X_2$ (the set of arcs $[Y,X_2]$) and the arcs that have one end in $X_1$ and the other end in $X_2$ (the set of arcs $[X_1,X_2]$) the label of each arc is distinct.
We relax this requirement in the following manner.
The set of all arcs in $[X_1,Y]$ with the same label must arc-induce a clean bipartite graph (defined in Section~\ref{sec:term}) and the set of all arcs in $[Y,X_2]$ with the same label must arc-induce a clean bipartite graph. 
Furthermore, the only restriction on the labels of the arcs in $[X_1,X_2]$ is that the arcs of $[X_1,X_2]$ must not have a label identical to a label of any of the arcs of $A(G)\backslash [X_1,X_2]$.

The rest of the paper is organised as follows.
In the next sections, we introduce new definitions that are necessary due to the relaxation of the two decomposition theorems.
Furthermore, we recall the formal graph definitions (in Section~\ref{sec:term}), the definition of the VRSP as well as the graph-decomposition theorems, together with other relevant terminology and notation (in Section~\ref{Terminology_products}), the notions of graph isomorphism and contraction of labelled acyclic directed multigraphs (in Section~\ref{Terminology_morphisms}), and the two graph theorems given in~\cite{dam} (in Section~\ref{twographtheorems}).
We relax the two decomposition theorems from \cite{dam} and state and proof a lemma decomposing bipartite graphs. 
We use the VRSP, the new lemma and the two decomposition theorems to state and prove the two (relaxed) decomposition theorems (in Section~\ref{sec:decomprev}).

\section{Terminology and notation}\label{sec:term}
We use the textbook of Bondy and Murty \cite{GraphTheory} for terminology and notation we do not specify here. 
Throughout, unless we specify explicitly that we consider other types of graphs, all graphs we consider are {\em edge-labelled acyclic directed multigraphs\/}, i.e., they may have multiple labelled arcs. Such graphs consist of a {\em vertex set\/} $V$ (representing the states of a process), an {\em arc set\/}  $A$ (representing the actions, i.e., transitions from one state to another), a set of {\em labels\/} $L$ (in our applications, a set of label pairs, each representing a type of action and the worst-case duration of its execution), and two mappings.
The first mapping $\mu: A\rightarrow V\times V$ is an incidence function that identifies the {\em tail\/} and {\em head\/} of each arc $a\in A$. 
In particular, $\mu(a)=(u,v)$ means that the arc $a$ is directed from $u\in V$ to $v\in V$, where $tail(a)=u$ and $head(a)=v$. We also call $u$ and $v$ the {\em ends\/} of $a$. 
The second mapping $\lambda :A\rightarrow L$ assigns a label pair $\lambda(a)=(\ell(a),t(a))$ to each arc $a\in A$, where $\ell(a)$ is a string representing the (name of an) action and $t(a)$ is the {\em weight\/} of the arc $a$.
This weight $t(a)$ is a real positive number representing the worst-case execution time of the action represented by $\ell(a)$.

Let $G$ denote a graph according to the above definition.
An arc $a\in A(G)$ is called an {\em in-arc\/} of $v\in V(G)$ if $head(a)=v$, and an {\em out-arc\/} of $v$ if $tail(a)=v$. The {\em in-degree\/} of $v$, denoted by $d^-(v)$, is the number of in-arcs of $v$ in $G$; the {\em out-degree\/} of $v$, denoted by $d^+(v)$, is the number of out-arcs of $v$ in $G$.
The subset of $V(G)$ consisting of vertices $v$ with $d^-(v)=0$ is called the {\em source\/} of $G$, and is denoted by $S'(G)$. The subset of $V(G)$ consisting of vertices $v$ with $d^+(v)=0$ is called the {\em sink\/} of $G$, and is denoted by $S''(G)$.

For disjoint nonempty sets $X,Y\subseteq V(G)$, $[X,Y]$ denotes the set of arcs of $G$ with one end in $X$ and one end in $Y$. If the head of the arc $a\in [X,Y]$ is in $Y$, we call $a$ a {\em forward arc\/} (of $[X,Y]$); otherwise, we call it a {\em backward arc\/}. 

The acyclicity of $G$ implies a natural ordering of the vertices into disjoint sets, as follows.
We define $S^0(G)$ to denote the set of vertices with in-degree 0 in $G$ (so $S^0(G)=S'(G)$), $S^1(G)$ the set of vertices with in-degree 0 in the graph obtained from $G$ by deleting the vertices of $S^0(G)$ and all arcs with tails in $S^0(G)$, and so on, until the final set $S^{t}(G)$ contains the remaining vertices with in-degree 0 and out-degree 0 in the remaining graph. Note that these sets are well-defined since $G$ is acyclic, and also note that $S^{t}(G)\neq S''(G)$, in general.  
If a vertex $v\in V(G)$ is in the set $S^j(G)$ in the above ordering, we say that $v$ is {\em at $level\, j$\/} in $G$.

A graph $G$ is called {\em weakly connected\/} if all pairs of distinct vertices $u$ and $v$ of $G$ are connected through a 
sequence of distinct vertices $u=v_0v_1\ldots  v_k=v$ and arcs $a_1a_2\ldots a_k$ of $G$ with $\mu(a_i) = (v_{i-1}, v_i)$ or $(v_{i},v_{i-1})$ for $i=1,2,\ldots ,k$. 
We are mainly interested in weakly connected graphs, or in the weakly connected components of a graph $G$. 
If $X\subseteq V(G)$, then the {\em subgraph of $G$ induced by $X$\/}, denoted as $G[X]$, is the graph on the vertex set $X$ containing all the arcs of $G$ which have both their ends in $X$ (together with $L$, $\mu$ and $\lambda$ restricted to this subset of the arcs). If $X\subseteq V$ induces a weakly connected subgraph of $G$, but there is no set $Y\subseteq V$ such that $G[Y]$ is weakly connected and $X$ is a proper subset of $Y$, then $G[X]$ is called a {\em weakly connected component\/} of $G$. 
If $X\subseteq A(G)$, then the {\em subgraph of $G$ arc-induced by $X$\/}, denoted as $G\{X\}$, is the graph on arc set $X$ containing all the vertices of $G$ which are an end of an arc in $X$ (together with $L$, $\mu$ and $\lambda$ restricted to this subset of the arcs). 

A subset $A'$ of arcs $a\in A$ with $\lambda(a)=\lambda_1$ is called the largest subset of arcs with the same label pair $\lambda_1$ if there does not exist an arc $b\in A\backslash A'$ with $\lambda(b)=\lambda_2$ and $\lambda_1=\lambda_2$.

In the sequel, throughout we omit the words weakly connected, so a component should always be understood as a weakly connected component. In contrast to the notation in the textbook of Bondy and Murty \cite{GraphTheory},
we use $\omega(G)$ to denote the number of components of a graph $G$.

We denote the components of $G$ by $G_i$, where $i$ ranges from 1 to $\omega(G)$.
In that case, we use $V_i$, $A_i$ and $L_i$ as a shorthand notation for $V(G_i)$, $A(G_i)$ and $L(G_i)$, respectively.
The mappings $\mu$ and $\lambda$ have natural counterparts restricted to the subsets $A_i\subset A(G)$ that we do not specify explicitly. 
We use $G=\sum\limits_{i=1}^{\omega(G)} G_i$ to indicate that $G$ is the disjoint union of its components, implicitly defining its components as $G_1$ up to $G_{\omega(G)}$. In particular, $G=G_1$ if and only if $G$ is weakly connected itself.
Furthermore, we use $\overundercup{i=1}{\omega(G)} G_i$ to denote the graph with vertex set $\overundercup{i=1}{\omega(G)} V_i$, arc set $\overundercup{i=1}{\omega(G)} A_i$ with the mappings $\mu_i(a_i)=(u_i,v_i)$ and $\lambda(a_i)=(\ell(a_i),t(a_i))$ for each arc $a_i\in A_i$.

A graph $G$ according to the above definition is called \emph{bipartite} if there exists a partition of nonempty sets $V_1$ and $V_2$ of $V(G)$ into two partite sets (i.e., $V(G) = V_1 \cup V_2$, $V_1 \cap V_2 = \emptyset$) such that every arc of $G$ has its head vertex and tail vertex in different partite sets. Such a graph is called a \emph{bipartite graph}, and we denote such a bipartite graph $G$ by $B(V_1, V_2)$. 
A bipartite graph $B(V_1, V_2)$ is called complete if, for every pair $x \in V_1$, $y \in V_2$, there is an arc $a$ with $\mu(a)=(x,y)$ or $\mu(a)=(y,x)$ in $B(V_1, V_2)$.
We call $B(V_1, V_2)$ a \emph{trivial bipartite graph} if $|V_1|=|V_2|=1$ and $|A(B(V_1,V_2))|\geq 1$.
Finally, we call a bipartite graph $B(V_1, V_2)$ a \emph{clean bipartite graph} if all subgraphs $B(V'_1,V'_2)$ of $B(V_1, V_2)$ are complete, where each subgraph $B(V'_1,V'_2)$ is arc-induced by all arcs in $[V_1,V_2]$ with the same label pair,and, $[V_1,V_2]$ has no backward arcs or $[V_1,V_2]$ has no forward arcs.

\section{Graph products}\label{Terminology_products}
In this section, we define the three graph products we are using for our decomposition theorems. 
Instead of defining products for general pairs of graphs, 
for notational reasons we find it convenient to define those products for two components $G_i$ and $G_j$ of a disconnected graph $G$. 
But, before we define the Cartesian product $G_i\,\Box\, G_j$, the intermediate product $G_i\boxtimes G_j$ and the VRSP $G_i\boxbackslash G_j$ of $G_i$ and $G_j$, we have to define the notion of an (a)synchronous arc.
Therefore, an arc $a\in A_i$ with label pair $\lambda(a)$ is a \emph{synchronising arc} with respect to $G_j$, if and only if there exists an arc $b\in A_j$ with label pair $\lambda(b)$ such that $\lambda(a)=\lambda(b)$.
Furthermore, an arc $a$ with label pair $\lambda(a)$ of $G_i\boxtimes G_j$ or $G_i\boxbackslash G_j$ (the graph products $\boxtimes$ and $\boxbackslash$ are defined in the sequel)  is a \emph{synchronous} arc, whenever there exist a pair of arcs $a_i\in A_i$ and $a_j\in A_j$ with $\lambda(a)=\lambda(a_i)=\lambda(a_j)$.
Analogously, an arc $a$ with label pair $\lambda(a)$ of $G_i\boxtimes G_j$ or $G_i\boxbackslash G_j$ is an \emph{asynchronous} arc, whenever $\lambda(a)\notin L_i$ or $\lambda(a)\notin L_j$.

But first, in Figure~\ref{Example0}, we give an example of the three products. 
At the top and the left of Figure~\ref{Example0}, we have the two graphs $G_i$ and $G_j$.
Then, in the middle, we have the Cartesian product of $G_i$ and $G_j$, $G_i\,\Box\, G_j$. 
On the right, we have the intermediate product of $G_i$ and $G_j$, $G_i\,\boxtimes\, G_j$. Here we see that the asynchronous arcs with label pairs not equal to $s$ of $G_i\,\Box\, G_j$ are maintained, whereas the synchronous arcs with label pair $s$ are replaced by one arc with label pair $s$.
At the bottom, we have the VRSP of $G_i$ and $G_j$, $G_i\,\boxbackslash\, G_j$ and here we see that the vertices with $level\,\,  0$ in $G_i\boxtimes G_j$ and $level > 0$ in $G_i\,\Box\, G_j$ are removed.
Note, in the iterations, first, the vertices $(u_3,v_1)$ and $(u_1,v_3)$, and their arcs, are removed from $G_i\boxtimes G_j$.
In a second iteration, this is followed by the removal of the vertices $(u_4,v_1), (u_3,v_2), (u_2,v_3)$ and $(u_1,v_4)$, and their arcs, because these vertices have $level\,\, 0$ due to the removal of $(u_3,v_1)$ and $(u_1,v_3)$.
In the third and last iteration, the vertices $(u_4,v_2)$ and $(u_2,v_4)$ are removed, leading to the graph $G_i\boxbackslash G_j$.
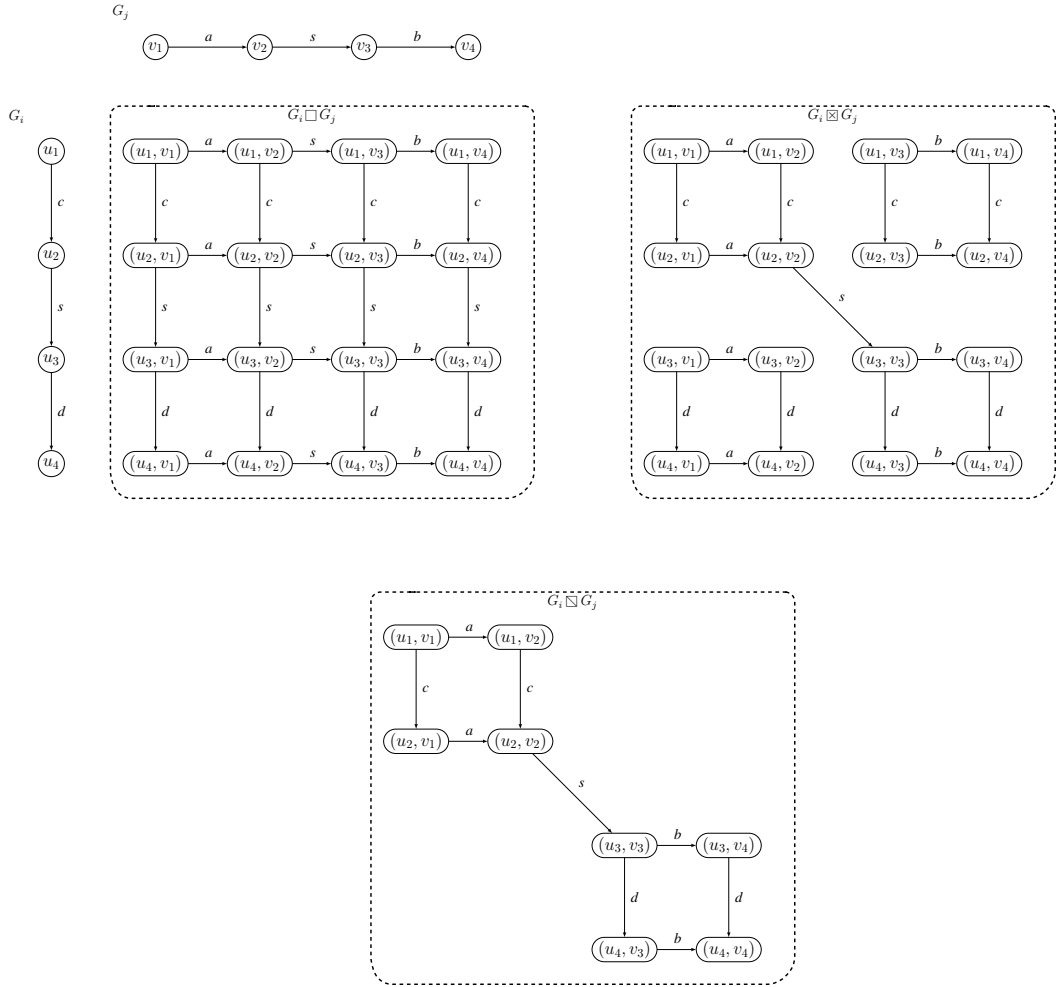
\begin{figure}[H]
\begin{center}
\resizebox{0.85\textwidth}{!}{
\begin{tikzpicture}[->,>=latex,shorten >=0pt,auto,node distance=2.5cm,
  main node/.style={circle,fill=blue!10,draw, font=\sffamily\Large\bfseries}]
  \tikzset{VertexStyle/.append style={
  font=\itshape\large, shape = circle,inner sep = 2pt, outer sep = 0pt,minimum size = 20 pt,draw}}
  \tikzset{EdgeStyle/.append style={thin}}
  \tikzset{LabelStyle/.append style={font = \itshape}}
  \SetVertexMath

  \def\x{0.0}
  \def\y{0.0}
  \node at (\x-1,\y-2) {$G_i$};
  \node at (\x+2,\y+1) {$G_j$};

  \Vertex[x=\x-0, y=\y-3.0]{u_1}
  \Vertex[x=\x-0, y=\y-6.0]{u_2}
  \Vertex[x=\x-0, y=\y-9.0]{u_3}
  \Vertex[x=\x-0, y=\y-12.0]{u_4}

\Edge[labelstyle={xshift=1pt, yshift=0pt},label = c](u_1)(u_2)
\Edge[labelstyle={xshift=1pt, yshift=0pt},label = s](u_2)(u_3)
\Edge[labelstyle={xshift=1pt, yshift=0pt},label = d](u_3)(u_4)

  \Vertex[x=\x+3, y=\y]{v_1}
  \Vertex[x=\x+6, y=\y]{v_2}
  \Vertex[x=\x+9, y=\y]{v_3}
  \Vertex[x=\x+12, y=\y]{v_4}

\Edge[labelstyle={xshift=1pt, yshift=1pt},label = a](v_1)(v_2)
\Edge[labelstyle={xshift=1pt, yshift=1pt},label = s](v_2)(v_3)
\Edge[labelstyle={xshift=1pt, yshift=1pt},label = b](v_3)(v_4)

\tikzset{VertexStyle/.append style={
  font=\itshape\large, shape = rounded rectangle, inner sep = 2pt, outer sep = 0pt,minimum size = 20 pt,draw}}

  \def\x{3.0}
  \def\y{0.0}
  \node at (\x+4.5,\y-2) {$G_i\,\Box\,G_j$};

  \Vertex[x=\x-0, y=\y-3.0, L={(u_1,v_1)}]{u_1v_1}
  \Vertex[x=\x-0, y=\y-6.0, L={(u_2,v_1)}]{u_2v_1}
  \Vertex[x=\x-0, y=\y-9.0, L={(u_3,v_1)}]{u_3v_1}
  \Vertex[x=\x-0, y=\y-12.0, L={(u_4,v_1)}]{u_4v_1}
  \Vertex[x=\x+3, y=\y-3.0, L={(u_1,v_2)}]{u_1v_2}
  \Vertex[x=\x+3, y=\y-6.0, L={(u_2,v_2)}]{u_2v_2}
  \Vertex[x=\x+3, y=\y-9.0, L={(u_3,v_2)}]{u_3v_2}
  \Vertex[x=\x+3, y=\y-12.0, L={(u_4,v_2)}]{u_4v_2}
  \Vertex[x=\x+6, y=\y-3.0, L={(u_1,v_3)}]{u_1v_3}
  \Vertex[x=\x+6, y=\y-6.0, L={(u_2,v_3)}]{u_2v_3}
  \Vertex[x=\x+6, y=\y-9.0, L={(u_3,v_3)}]{u_3v_3}
  \Vertex[x=\x+6, y=\y-12.0, L={(u_4,v_3)}]{u_4v_3}
  \Vertex[x=\x+9, y=\y-3.0, L={(u_1,v_4)}]{u_1v_4}
  \Vertex[x=\x+9, y=\y-6.0, L={(u_2,v_4)}]{u_2v_4}
  \Vertex[x=\x+9, y=\y-9.0, L={(u_3,v_4)}]{u_3v_4}
  \Vertex[x=\x+9, y=\y-12.0, L={(u_4,v_4)}]{u_4v_4}

\Edge[labelstyle={xshift=1pt, yshift=1pt},label = a](u_1v_1)(u_1v_2)
\Edge[labelstyle={xshift=1pt, yshift=1pt},label = a](u_2v_1)(u_2v_2)
\Edge[labelstyle={xshift=1pt, yshift=1pt},label = a](u_3v_1)(u_3v_2)
\Edge[labelstyle={xshift=1pt, yshift=1pt},label = a](u_4v_1)(u_4v_2)
\Edge[labelstyle={xshift=1pt, yshift=1pt},label = s](u_1v_2)(u_1v_3)
\Edge[labelstyle={xshift=1pt, yshift=1pt},label = s](u_2v_2)(u_2v_3)
\Edge[labelstyle={xshift=1pt, yshift=1pt},label = s](u_3v_2)(u_3v_3)
\Edge[labelstyle={xshift=1pt, yshift=1pt},label = s](u_4v_2)(u_4v_3)
\Edge[labelstyle={xshift=1pt, yshift=1pt},label = b](u_1v_3)(u_1v_4)
\Edge[labelstyle={xshift=1pt, yshift=1pt},label = b](u_2v_3)(u_2v_4)
\Edge[labelstyle={xshift=1pt, yshift=1pt},label = b](u_3v_3)(u_3v_4)
\Edge[labelstyle={xshift=1pt, yshift=1pt},label = b](u_4v_3)(u_4v_4)

\Edge[labelstyle={xshift=1pt, yshift=0pt},label = c](u_1v_1)(u_2v_1)
\Edge[labelstyle={xshift=1pt, yshift=0pt},label = s](u_2v_1)(u_3v_1)
\Edge[labelstyle={xshift=1pt, yshift=0pt},label = d](u_3v_1)(u_4v_1)
\Edge[labelstyle={xshift=1pt, yshift=0pt},label = c](u_1v_2)(u_2v_2)
\Edge[labelstyle={xshift=1pt, yshift=0pt},label = s](u_2v_2)(u_3v_2)
\Edge[labelstyle={xshift=1pt, yshift=0pt},label = d](u_3v_2)(u_4v_2)
\Edge[labelstyle={xshift=1pt, yshift=0pt},label = c](u_1v_3)(u_2v_3)
\Edge[labelstyle={xshift=1pt, yshift=0pt},label = s](u_2v_3)(u_3v_3)
\Edge[labelstyle={xshift=1pt, yshift=0pt},label = d](u_3v_3)(u_4v_3)
\Edge[labelstyle={xshift=1pt, yshift=0pt},label = c](u_1v_4)(u_2v_4)
\Edge[labelstyle={xshift=1pt, yshift=0pt},label = s](u_2v_4)(u_3v_4)
\Edge[labelstyle={xshift=1pt, yshift=0pt},label = d](u_3v_4)(u_4v_4)

\draw[circle, -,dashed, very thick,rounded corners=8pt] (\x-0.2,\y-1.7)--(\x+10.4,\y-1.7) --(\x+10.9,\y-1.7) -- (\x+10.9,\y-12.5)-- (\x+10.4,\y-13) -- (\x-1.0,\y-13) -- (\x-1.3,\y-12.5) -- (\x-1.3,\y-1.7) -- (\x-0.3,\y-1.7)--(\x-0.1,\y-1.7);

  \def\x{18.0}
  \def\y{0.0}
  \node at (\x+4.5,\y-2) {$G_i\boxtimes G_j$};
  \Vertex[x=\x-0, y=\y-3.0, L={(u_1,v_1)}]{u_1v_1}
  \Vertex[x=\x-0, y=\y-6.0, L={(u_2,v_1)}]{u_2v_1}
  \Vertex[x=\x-0, y=\y-9.0, L={(u_3,v_1)}]{u_3v_1}
  \Vertex[x=\x-0, y=\y-12.0, L={(u_4,v_1)}]{u_4v_1}
  \Vertex[x=\x+3, y=\y-3.0, L={(u_1,v_2)}]{u_1v_2}
  \Vertex[x=\x+3, y=\y-6.0, L={(u_2,v_2)}]{u_2v_2}
  \Vertex[x=\x+3, y=\y-9.0, L={(u_3,v_2)}]{u_3v_2}
  \Vertex[x=\x+3, y=\y-12.0, L={(u_4,v_2)}]{u_4v_2}
  \Vertex[x=\x+6, y=\y-3.0, L={(u_1,v_3)}]{u_1v_3}
  \Vertex[x=\x+6, y=\y-6.0, L={(u_2,v_3)}]{u_2v_3}
  \Vertex[x=\x+6, y=\y-9.0, L={(u_3,v_3)}]{u_3v_3}
  \Vertex[x=\x+6, y=\y-12.0, L={(u_4,v_3)}]{u_4v_3}
  \Vertex[x=\x+9, y=\y-3.0, L={(u_1,v_4)}]{u_1v_4}
  \Vertex[x=\x+9, y=\y-6.0, L={(u_2,v_4)}]{u_2v_4}
  \Vertex[x=\x+9, y=\y-9.0, L={(u_3,v_4)}]{u_3v_4}
  \Vertex[x=\x+9, y=\y-12.0, L={(u_4,v_4)}]{u_4v_4}

\Edge[labelstyle={xshift=1pt, yshift=1pt},label = a](u_1v_1)(u_1v_2)
\Edge[labelstyle={xshift=1pt, yshift=1pt},label = a](u_2v_1)(u_2v_2)
\Edge[labelstyle={xshift=1pt, yshift=1pt},label = a](u_3v_1)(u_3v_2)
\Edge[labelstyle={xshift=1pt, yshift=1pt},label = a](u_4v_1)(u_4v_2)
\Edge[labelstyle={xshift=1pt, yshift=1pt},label = s](u_2v_2)(u_3v_3)
\Edge[labelstyle={xshift=1pt, yshift=1pt},label = b](u_1v_3)(u_1v_4)
\Edge[labelstyle={xshift=1pt, yshift=1pt},label = b](u_2v_3)(u_2v_4)
\Edge[labelstyle={xshift=1pt, yshift=1pt},label = b](u_3v_3)(u_3v_4)
\Edge[labelstyle={xshift=1pt, yshift=1pt},label = b](u_4v_3)(u_4v_4)

\Edge[labelstyle={xshift=1pt, yshift=0pt},label = c](u_1v_1)(u_2v_1)
\Edge[labelstyle={xshift=1pt, yshift=0pt},label = d](u_3v_1)(u_4v_1)
\Edge[labelstyle={xshift=1pt, yshift=0pt},label = c](u_1v_2)(u_2v_2)
\Edge[labelstyle={xshift=1pt, yshift=0pt},label = d](u_3v_2)(u_4v_2)
\Edge[labelstyle={xshift=1pt, yshift=0pt},label = c](u_1v_3)(u_2v_3)
\Edge[labelstyle={xshift=1pt, yshift=0pt},label = d](u_3v_3)(u_4v_3)
\Edge[labelstyle={xshift=1pt, yshift=0pt},label = c](u_1v_4)(u_2v_4)
\Edge[labelstyle={xshift=1pt, yshift=0pt},label = d](u_3v_4)(u_4v_4)

\draw[circle, -,dashed, very thick,rounded corners=8pt] (\x-0.2,\y-1.7)--(\x+10.4,\y-1.7) --(\x+10.9,\y-1.7) -- (\x+10.9,\y-12.5)-- (\x+10.4,\y-13) -- (\x-1.0,\y-13) -- (\x-1.3,\y-12.5) -- (\x-1.3,\y-1.7) -- (\x-0.3,\y-1.7)--(\x-0.1,\y-1.7);
  \def\x{10.5}
  \def\y{-14.0}
  \node at (\x+4.5,\y-2) {$G_i\boxbackslash G_j$};
  \Vertex[x=\x-0, y=\y-3.0, L={(u_1,v_1)}]{u_1v_1}
  \Vertex[x=\x-0, y=\y-6.0, L={(u_2,v_1)}]{u_2v_1}
  \Vertex[x=\x+3, y=\y-3.0, L={(u_1,v_2)}]{u_1v_2}
  \Vertex[x=\x+3, y=\y-6.0, L={(u_2,v_2)}]{u_2v_2}
  \Vertex[x=\x+6, y=\y-9.0, L={(u_3,v_3)}]{u_3v_3}
  \Vertex[x=\x+6, y=\y-12.0, L={(u_4,v_3)}]{u_4v_3}
  \Vertex[x=\x+9, y=\y-9.0, L={(u_3,v_4)}]{u_3v_4}
  \Vertex[x=\x+9, y=\y-12.0, L={(u_4,v_4)}]{u_4v_4}

\Edge[labelstyle={xshift=1pt, yshift=1pt},label = a](u_1v_1)(u_1v_2)
\Edge[labelstyle={xshift=1pt, yshift=1pt},label = a](u_2v_1)(u_2v_2)
\Edge[labelstyle={xshift=1pt, yshift=1pt},label = s](u_2v_2)(u_3v_3)
\Edge[labelstyle={xshift=1pt, yshift=1pt},label = b](u_3v_3)(u_3v_4)
\Edge[labelstyle={xshift=1pt, yshift=1pt},label = b](u_4v_3)(u_4v_4)

\Edge[labelstyle={xshift=1pt, yshift=0pt},label = c](u_1v_1)(u_2v_1)
\Edge[labelstyle={xshift=1pt, yshift=0pt},label = c](u_1v_2)(u_2v_2)
\Edge[labelstyle={xshift=1pt, yshift=0pt},label = d](u_3v_3)(u_4v_3)
\Edge[labelstyle={xshift=1pt, yshift=0pt},label = d](u_3v_4)(u_4v_4)

\draw[circle, -,dashed, very thick,rounded corners=8pt] (\x-0.2,\y-1.7)--(\x+10.4,\y-1.7) --(\x+10.9,\y-1.7) -- (\x+10.9,\y-12.5)-- (\x+10.4,\y-13) -- (\x-1.0,\y-13) -- (\x-1.3,\y-12.5) -- (\x-1.3,\y-1.7) -- (\x-0.3,\y-1.7)--(\x-0.1,\y-1.7);

\end{tikzpicture}
}
\end{center}
\caption{The three products for the graphs $G_i$ and $G_j$, the Cartesian product $G_i\,\Box\, G_j$, the intermediate product, $G_i\,\boxtimes\, G_j$ and the VRSP $G_i\,\boxbackslash\, G_j$.}
  \label{Example0}
\end{figure}
We start by introducing the next analogue of the Cartesian product.

The {\em Cartesian product\/} $G_i\,\Box\, G_j$ of $G_i$ and $G_j$ is defined as the graph on the vertex set $V_{i,j}=V_i\times V_j$, and arc set $A_{i,j}$ consisting of two types of labelled arcs. 
For each arc $a\in A_i$ with $\mu(a)=(v_i,w_i)$, an {\em arc of type $i$\/} is introduced between tail $(v_i,v_j)\in V_{i,j}$ and head $(w_i,w_j)\in V_{i,j}$ whenever $v_j=w_j$; such an arc receives the label pair $\lambda(a)$. 
This implicitly defines parts of the mappings $\mu$ and $\lambda$ for $G_i \,\Box\, G_j$.
Similarly, for each arc $a\in A_j$ with $\mu(a)=(v_j,w_j)$, an {\em arc of type $j$\/} is introduced between tail $(v_i,v_j)\in V_{i,j}$ and head $(w_i,w_j)\in V_{i,j}$ whenever $v_i=w_i$; such an arc receives the label pair $\lambda(a)$. 
This completes the definition of $A_{i,j}$ and the  mappings $\mu$ and $\lambda$ for $G_i \,\Box\, G_j$.
So, arcs of type $i$ and $j$ correspond to arcs of $G_i$ and $G_j$, respectively, and have the associated label pairs.
For $k\ge 3$, the Cartesian product $G_1\,\Box\, G_2\,\Box \cdots\Box\, G_k$ is defined recursively as $((G_1\,\Box\, G_2)\,\Box \cdots )\Box\, G_k$. This Cartesian product is commutative and associative, as can be verified easily and is a well-known fact for the undirected analogue. 
Since we are particularly interested in synchronising arcs, we modify the Cartesian product $G_i\,\Box\, G_j$ according to the existence of synchronising arcs, i.e., pairs of arcs with the same label pair, with one arc in $G_i$ and one arc in $G_j$. 

The first step in this modification consists of ignoring (in fact deleting) the synchronising arcs while forming arcs in the product, but additionally combining pairs of synchronising arcs of $G_i$ and $G_j$ into one arc, yielding the {\em intermediate product\/} which we denote by $G_i \boxtimes G_j$. 

To be more precise, $G_i \boxtimes G_j$ is obtained from $G_i\,\Box\, G_j$ by first ignoring all except for the so-called {\em asynchronous\/} arcs, i.e., by only maintaining all arcs $a\in A_{i,j}$ for which $\mu(a)=((v_i,v_j), (w_i,$ $w_j))$, whenever $v_j=w_j$ and $\lambda(a) \notin L_j$, as well as all arcs $a\in A_{i,j}$ for which $\mu(a)=((v_i,v_j),(w_i,$ $w_j))$, whenever $v_i=w_i$ and $\lambda(a)\notin L_i$. 
Additionally, we add arcs that replace synchronising pairs $a_i\in A_i$ and $a_j\in A_j$ with $\lambda(a_i)=\lambda(a_j)$. If $\mu(a_i)=(v_i,w_i)$ and $\mu(a_j)=(v_j,w_j)$, such a pair is replaced by an arc $a_{i,j}$ with  $\mu(a_{i,j})=((v_i,v_j),(w_i,w_j))$ and $\lambda(a_{i,j})=\lambda(a_i)$. 
We call such arcs  of $G_i \boxtimes G_j$ {\em synchronous\/} arcs. 

The second step in this modification consists of removing (from $G_i \boxtimes G_j$) the vertices $(v_i,v_j)\in V_{i,j}$ and the arcs $a$ with $tail(a)=(v_i,v_j)$, in the case that $(v_i,v_j)$ has $level > 0$ in $G_i\,\Box\, G_j$ but $level\, 0$ in $G_i\boxtimes G_j$. 
This is then repeated in the newly obtained graph, and so on, until there are no more vertices at $level\, 0$ in the current graph that are at $level > 0$ in $G_i\,\Box\, G_j$. This finds its motivation in the fact that in our applications, the states that are represented by such vertices can never be reached, so are irrelevant. 

The resulting graph is called the {\em vertex-removing synchronised product\/} (VRSP for short) of $G_i$ and $G_j$, and denoted as $G_i \boxbackslash G_j$. 
For $k\ge 3$, the {VRSP} $G_1 \boxbackslash G_2 \boxbackslash \cdots \boxbackslash G_k$ is defined recursively as $((G_1 \boxbackslash G_2) \boxbackslash \cdots ) \boxbackslash G_k$. The VRSP is commutative, but not associative in general, in contrast to the Cartesian product. These properties are not relevant for the decomposition results that follow. 
However, for these results, it is relevant to introduce counterparts of graph isomorphism and graph contraction that apply to our types of graphs. We define these counterparts in the next section.

\section{Graph isomorphism and graph contraction}\label{Terminology_morphisms}
The isomorphism we introduce in this section is an analogue of a known concept for unlabelled graphs, but involves statements on the labels.

We assume that two different arcs with the same head and tail have different label pairs; otherwise, we replace such multiple arcs by one arc with that label pair, because these arcs represent exactly the same action at the same stage of a process.

Formally, an {\em isomorphism\/} from a graph $G$ to a graph $H$ consists of two bijections $\phi : V(G)\rightarrow V(H)$ and $\rho : A(G)\rightarrow A(H)$ such that for all $a \in A(G)$, one has $\mu(a) = (u, v)$ if and only if $\mu(\rho(a))=(\phi(u),\phi(v))$ and $\lambda(a)=\lambda(\rho(a))$. 
Since we assume that two different arcs with the same head and tail have different label pairs, however, the bijection $\rho$ is superfluous. The reason is that, if $(\phi,\rho)$ is an isomorphism, then $\rho$ is completely determined by $\phi$ and the label pairs. 
In fact, if $(\phi,\rho)$ is an isomorphism and $\mu(a)=(u,v)$ for an arc $a\in A(G)$, then $\rho(a)$ is the unique arc $b\in A(H)$ with $\mu(b)=(\phi(u),\phi(v))$ and  label pair $\lambda(b)=\lambda(a)$. 
Thus, we may define an isomorphism from $G$ to $H$ as a bijection $\phi : V(G)\rightarrow V(H)$ such that there exists an arc $a \in A(G)$ with $\mu(a)=(u,v)$ if and only if there exists an arc $b\in A(H)$ with $\mu(b)=(\phi(u),\phi(v))$ and $\lambda(b)=\lambda(a)$.
An isomorphism from $G$ to $H$ is denoted as $G\cong H$.

Next, we define what we mean by contraction.  
Let $X$ be a nonempty proper subset of $V(G)$, and let $Y=V(G)\setminus X$. 
By {\em contracting $X$\/} we mean replacing $X$ by a new vertex $\tilde{x}$, deleting all arcs with both ends in $X$, replacing each arc $a\in A(G)$ with $\mu(a)=(u,v)$ for $u\in X$ and $v\in Y$ by an arc $c$ with $\mu(c)=(\tilde{x},v)$ and $\lambda (c)=\lambda(a)$, and replacing each arc $b\in A(G)$ with $\mu(b)=(u,v)$ for $u\in Y$ and $v\in X$ by an arc $d$ with $\mu(d)=(u,\tilde{x})$ and $\lambda (d)=\lambda(b)$. We denote the resulting graph as $G/X$, and say that $G/X$ is the {\em contraction of $G$ with respect to $X$\/}. 
If we contract more than one subset $X_i$ of $V$ we denote $((G/X_1)/X_2\ldots)/X_n$ by $G/X_1/X_2\ldots /X_n$.

\section{Graph theorems from~\cite{dam}}\label{twographtheorems}
Finally, we recall the two decomposition theorems that were stated and proved in~\cite{dam}. 
\begin{theorem}[\cite{dam}]\label{theorem_1}
Let $G$ be a graph, let $X$ be a nonempty proper subset of $V(G)$, and let $Y=V(G)\setminus X$. Suppose that all the arcs of $[X,Y]$ have distinct label pairs and that the arcs of $G/X$ and $G/Y$ corresponding to the arcs of $[X,Y]$ are the only synchronising arcs of $G/X$ and $G/Y$. 
If $S'(G)\subseteq X$ and $[X,Y]$ has no backward arcs, then 
 $G\cong G/Y\boxbackslash G/X$.
\end{theorem}
\begin{theorem}[\cite{dam}]\label{theorem_2}
Let $G$ be a graph, and let $X_1$, $X_2$ and $Y=V(G)\setminus (X_1\cup X_2)$ be three disjoint nonempty subsets of $V(G)$. Suppose that all the arcs of $[X_1,Y]$ have distinct label pairs, all the arcs of $[Y,X_2]$ have distinct label pairs, all the arcs of $[X_1,X_2]$ have distinct label pairs, the arcs of $[X_1,X_2]$ have no label pairs in common with any arcs in $[X_1,Y]\cup[Y,X_2]$, and that the arcs of $G/X_1/X_2$ and $G/Y$ corresponding to the arcs of $[X_1,Y]\cup [Y,X_2]\cup [X_1,X_2]$ are the only synchronising arcs of $G/X_1/X_2$ and $G/Y$. 
If $S'(G)\subseteq X_1$, and $[X_1,Y]$, $[Y,X_2]$ and $[X_1,X_2]$ have no backward arcs, then 
 $G\cong G/Y\boxbackslash G/X_1/X_2$.
\end{theorem}

\section{New results}~\label{sec:decomprev}
We start with relaxing the requirement in Theorem~\ref{theorem_1} that states that all arcs of $[X,Y]$ have distinct label pairs in the following manner: each largest set of arcs of $[X,Y]$ with the same label pair arc-induces a complete bipartite subgraph of $G$.
Hence, $G\{[X,Y]\}$ is a clean bipartite subgraph of $G$.
Furthermore, we relax the requirement in Theorem~\ref{theorem_2} that all arcs of $[X_1,Y]$, $[Y,X_2]$ and $[X_1,X_2]$ have distinct label pairs in the following manner: firstly, each largest set of arcs of $[X_1,Y]$ with the same label pair arc-induces a complete bipartite subgraph of $G$, secondly, each largest set of arcs of $[Y,X_2]$ with the same label pair arc-induces a complete bipartite subgraph of $G$ and, thirdly, the label pairs of the arcs in $[X_1,X_2]$ do not have to be distinct.
Hence, $G\{[X_1,Y]\}$ is a clean bipartite subgraph of $G$ and $G\{[Y,X_2]\}$ is a clean bipartite subgraph of $G$.

The relaxed requirement of Theorem~\ref{theorem_1} and the first and second relaxed requirement of Theorem~\ref{theorem_2} are based on the decomposition of a complete bipartite graph where all arcs have the same label pair. 
The~third relaxed requirement of Theorem~\ref{theorem_2} is based on the observation that the contraction of $X_1$ and $X_2$, $G/X_1/X_2$, replaces the set of arcs $[X_1,X_2]$ by a set of arcs $[\{\tilde{x}_1\},\{\tilde{x}_2\}]$.
Hence, let $G'$ be the subgraph of $G/Y$ arc-induced by the set of arcs $[X_1,X_2]$ of $G/Y$ and let $G''$ be the subgraph of $G/X_1/X_2$ arc-induced by the set of arcs $[\{\tilde{x}_1\},\{\tilde{x}_2\}]$ of $G/X_1/X_2$.
Then the VRSP of $G'$ and $G''$ is isomorphic to $G'
$, i.e. $G'\cong G'\boxbackslash G''$.

We have depicted a simple example in Figure~\ref{Example1} which illustrates these three relaxed requirements.
At the upper left of Figure~\ref{Example1}, we show the graph $G$. 
The subgraph arc-induced by the arcs with label pair $c$ contains two complete bipartite subgraphs.
The arcs with label pair $c$ are the only arcs in $[X_1,Y]\cup [Y,X_2]$.
For all other sets of arcs in $G$ with the same label pair we do not require that these sets arc-induce a complete bipartite graph as they are not in $[X_1,Y]\cup [Y,X_2]$. 
At the lower left and the upper right of Figure~\ref{Example1}, we show the contracted graphs $G/Y$ and $G/X_1/X_2$, respectively.
At the lower right of Figure~\ref{Example1}, we show the intermediate product of the graphs $G/Y$ and $G/X_1/X_2$, $G/Y\boxtimes G/X_1/X_2$.
The vertices in the set $Z$ at the lower right of Figure~\ref{Example1} induce the graph $G/Y\boxbackslash G/X_1/X_2$ which is isomorphic to $G$. 

\begin{figure}[H]
\begin{center}
\resizebox{1.0\textwidth}{!}{
\begin{tikzpicture}[->,>=latex,shorten >=0pt,auto,node distance=2.5cm,
  main node/.style={circle,fill=blue!10,draw, font=\sffamily\Large\bfseries}]
  \tikzset{VertexStyle/.append style={
  font=\itshape\large, shape = circle,inner sep = 2pt, outer sep = 0pt,minimum size = 20 pt,draw}}
  \tikzset{EdgeStyle/.append style={thin}}
  \tikzset{LabelStyle/.append style={font = \itshape}}
  \SetVertexMath

  \clip (-18.75,-1) rectangle (13,-38);
  \def\x{0.0}
  \def\y{0.0}
\node at (\x-17.0,\y-2) {$G$};
\node at (\x-18.0,\y-3) {$X_1$};
\node at (\x-8,\y-3) {$X_2$};
\node at (\x-13,\y-3) {$Y$};
\node at (\x+1.0-2,\y-2) {$G/X_1/X_2$};
\node at (\x-14.0,\y-14) {$G/Y$};
\node at (\x+2.0,\y-9.5) {$G/Y\boxtimes G/X_1/X_2$};
\node at (\x-4,\y-13.6) {$Z$};
  \def\x{-18.0}
  \def\y{-6.0}
  \Vertex[x=\x-0, y=\y+1.0]{u_1}
  \Vertex[x=\x+2, y=\y+2.0]{u_2}
  \Vertex[x=\x+2, y=\y+0.0,L={u_{3}}]{u_3}
  \Vertex[x=\x+4, y=\y+2.0,L={u_{4}}]{u_4}
  \Vertex[x=\x+4, y=\y+0.0,L={u_{5}}]{u_5}
  \Vertex[x=\x+6, y=\y+2.0,L={u_{6}}]{u_6}
  \Vertex[x=\x+6, y=\y+0.0,L={u_{7}}]{u_7}
  \Vertex[x=\x+8, y=\y+2.0,L={u_{8}}]{u_8}
  \Vertex[x=\x+8, y=\y+0.0,L={u_{9}}]{u_9}
  \Vertex[x=\x+10, y=\y+1.0,L={u_{10}}]{u_10}

\Edge[style={in=180, out=60},labelstyle={xshift=14pt, yshift=4pt},label = a](u_1)(u_2)
\Edge[style={in=180, out=-60},label = a](u_1)(u_3)
\Edge[style={in=135, out=45},labelstyle={xshift=0pt, yshift=1pt},label = b](u_2)(u_8)
\Edge[style={in=225, out=-45},labelstyle={xshift=0pt, yshift=1pt},label = b](u_3)(u_9)
\Edge[labelstyle={xshift=1pt, yshift=0pt},label = c](u_2)(u_4)
\Edge[labelstyle={xshift=-16pt, yshift=12pt},label = c](u_2)(u_5)
\Edge(u_2)(u_5)
\Edge[labelstyle={xshift=-8pt, yshift=-9pt},label = c](u_3)(u_4)
\Edge[labelstyle={xshift=1pt, yshift=0pt},label = c](u_3)(u_5)

\Edge[labelstyle={xshift=1pt, yshift=0pt},label = d](u_4)(u_6)
\Edge[labelstyle={xshift=1pt, yshift=0pt},label = d](u_5)(u_7)
\Edge(u_4)(u_6)
\Edge(u_5)(u_7)

\Edge[labelstyle={xshift=1pt, yshift=0pt},label = c](u_6)(u_8)
\Edge[labelstyle={xshift=-16pt, yshift=12pt},label = c](u_6)(u_9)
\Edge(u_6)(u_9)
\Edge[labelstyle={xshift=-8pt, yshift=-9pt},label = c](u_7)(u_8)
\Edge[labelstyle={xshift=1pt, yshift=0pt},label = c](u_7)(u_9)
\Edge(u_6)(u_8)
\Edge(u_6)(u_9)
\Edge(u_7)(u_8)
\Edge(u_7)(u_9)

\Edge[style={in=120, out=0},labelstyle={xshift=0pt, yshift=4pt},label = f](u_8)(u_10)
\Edge[style={in=240, out=0},label = f](u_9)(u_10)

  \def\x{-18.0}
  \def\y{-15.0}
  \Vertex[x=\x+6, y=\y+0.0]{u_1}
  \Vertex[x=\x+6, y=\y-3.0]{u_2}
  \Vertex[x=\x+6, y=\y-6.0,L={u_{3}}]{u_3}
  \Vertex[x=\x+6, y=\y-9.0,L={\tilde{y}}]{u_4}
  \Vertex[x=\x+6, y=\y-12.0,L={u_{8}}]{u_8}
  \Vertex[x=\x+6, y=\y-15.0,L={u_{9}}]{u_9}
  \Vertex[x=\x+6, y=\y-18.0,L={u_{10}}]{u_10}

\Edge[style={in=135, out=225},labelstyle={xshift=14pt, yshift=4pt},label = a](u_1)(u_2)
\Edge[style={in=225-90, out=-225+90},label = a](u_1)(u_3)
\Edge[style={in=135, out=225},labelstyle={xshift=0pt, yshift=1pt},label = b](u_2)(u_8)
\Edge[style={in=225-90, out=-225+90},labelstyle={xshift=0pt, yshift=1pt},label = b](u_3)(u_9)
\Edge[style={in=225-90, out=-225+90},labelstyle={xshift=1pt, yshift=0pt},label = c](u_2)(u_4)
\Edge[labelstyle={xshift=1pt, yshift=0pt},label = c](u_3)(u_4)
\Edge(u_3)(u_4)
\Edge[labelstyle={xshift=1pt, yshift=0pt},label = c](u_4)(u_8)
\Edge(u_4)(u_8)
\Edge[style={in=135, out=225},labelstyle={xshift=-16pt, yshift=12pt},label = c](u_4)(u_9)
\Edge[style={in=135, out=225},labelstyle={xshift=0pt, yshift=4pt},label = f](u_8)(u_10)
\Edge[style={in=225-90, out=-225+90},label = f](u_9)(u_10)

  \def\x{-4.0}
  \def\y{-6.0}
  \Vertex[x=\x+0, y=\y+1,L={\tilde{x}_1}]{u_1}
  \Vertex[x=\x+15, y=\y+1,L={\tilde{x}_2}]{u_10}
  \Vertex[x=\x+3, y=\y+1,L={u_{4}}]{u_4}
  \Vertex[x=\x+6, y=\y+1,L={u_{5}}]{u_5}
  \Vertex[x=\x+9, y=\y+1,L={u_{6}}]{u_6}
  \Vertex[x=\x+12, y=\y+1,L={u_{7}}]{u_7}

\Edge[style={in=45+180, out=315},labelstyle={xshift=0pt, yshift=1pt},label = b](u_1)(u_10)
\Edge[style={in=225, out=315}](u_1)(u_10)
\Edge[labelstyle={xshift=1pt, yshift=0pt},label = c](u_1)(u_4)
\Edge[style={in=135, out=45},labelstyle={xshift=1pt, yshift=0pt},label = c](u_1)(u_5)

\Edge[style={in=135, out=45},labelstyle={xshift=1pt, yshift=0pt},label = d](u_4)(u_6)
\Edge[style={in=225, out=315},labelstyle={xshift=1pt, yshift=0pt},label = d](u_5)(u_7)

\Edge[style={in=135, out=225-90},labelstyle={xshift=0pt, yshift=4pt},label = c](u_6)(u_10)
\Edge[label = c](u_7)(u_10)

  \tikzset{VertexStyle/.append style={
  font=\itshape\large, shape = rounded rectangle, inner sep = 2pt, outer sep = 0pt,minimum size = 20 pt,draw}}
  
  \def\x{-4.0}
  \def\y{-15.0}
  \Vertex[x=\x, y=\y+0.0,L={(u_1,\tilde{x}_1)}]{u_11}
  \Vertex[x=\x, y=\y-3.0,L={(u_2,\tilde{x}_1)}]{u_21}
  \Vertex[x=\x, y=\y-6.0,L={(u_3,\tilde{x}_1)}]{u_31}
  \Vertex[x=\x, y=\y-9.0,L={(\tilde{y},\tilde{x}_1)}]{u_41}
  \Vertex[x=\x, y=\y-12.0,L={(u_8,\tilde{x}_1)}]{u_51}
\Vertex[x=\x, y=\y-15.0,L={(u_9,\tilde{x}_1)}]{u_61}
\Vertex[x=\x, y=\y-18.0,L={(u_{10},\tilde{x}_1)}]{u_71}

  \def\x{-1.0}
  \Vertex[x=\x, y=\y+0.0,L={(u_1,u_4)}]{u_12}
  \Vertex[x=\x, y=\y-3.0,L={(u_2,u_4)}]{u_22}
  \Vertex[x=\x, y=\y-6.0,L={(u_3,u_4)}]{u_32}
  \Vertex[x=\x, y=\y-9.0,L={(\tilde{y},u_4)}]{u_42}
  \Vertex[x=\x, y=\y-12.0,L={(u_8,u_4)}]{u_52}
\Vertex[x=\x, y=\y-15.0,L={(u_9,u_4)}]{u_62}
\Vertex[x=\x, y=\y-18.0,L={(u_{10},u_4)}]{u_72}

  \def\x{2.0}
  \Vertex[x=\x, y=\y+0.0,L={(u_1,u_5)}]{u_13}
  \Vertex[x=\x, y=\y-3.0,L={(u_2,u_5)}]{u_23}
  \Vertex[x=\x, y=\y-6.0,L={(u_3,u_5)}]{u_33}
  \Vertex[x=\x, y=\y-9.0,L={(\tilde{y},u_5)}]{u_43}
  \Vertex[x=\x, y=\y-12.0,L={(u_8,u_5)}]{u_53}
\Vertex[x=\x, y=\y-15.0,L={(u_9,u_5)}]{u_63}
\Vertex[x=\x, y=\y-18.0,L={(u_{10},u_5)}]{u_73}

  \def\x{5.0}
  \Vertex[x=\x, y=\y+0.0,L={(u_1,u_6)}]{u_14}
  \Vertex[x=\x, y=\y-3.0,L={(u_2,u_6)}]{u_24}
  \Vertex[x=\x, y=\y-6.0,L={(u_3,u_6)}]{u_34}
  \Vertex[x=\x, y=\y-9.0,L={(\tilde{y},u_6)}]{u_44}
  \Vertex[x=\x, y=\y-12.0,L={(u_8,u_6)}]{u_54}
\Vertex[x=\x, y=\y-15.0,L={(u_9,u_6)}]{u_64}
\Vertex[x=\x, y=\y-18.0,L={(u_{10},u_6)}]{u_74}

  \def\x{8.0}
  \Vertex[x=\x, y=\y+0.0,L={(u_1,u_7)}]{u_15}
  \Vertex[x=\x, y=\y-3.0,L={(u_2,u_7)}]{u_25}
  \Vertex[x=\x, y=\y-6.0,L={(u_3,u_7)}]{u_35}
  \Vertex[x=\x, y=\y-9.0,L={(\tilde{y},u_7)}]{u_45}
  \Vertex[x=\x, y=\y-12.0,L={(u_8,u_7)}]{u_55}
\Vertex[x=\x, y=\y-15.0,L={(u_9,u_7)}]{u_65}
\Vertex[x=\x, y=\y-18.0,L={(u_{10},u_7)}]{u_75}

  \def\x{11.0}
  \Vertex[x=\x, y=\y+0.0,L={(u_1,\tilde{x}_2)}]{u_16}
  \Vertex[x=\x, y=\y-3.0,L={(u_2,\tilde{x}_2)}]{u_26}
  \Vertex[x=\x, y=\y-6.0,L={(u_3,\tilde{x}_2)}]{u_36}
  \Vertex[x=\x, y=\y-9.0,L={(\tilde{y},\tilde{x}_2)}]{u_46}
  \Vertex[x=\x, y=\y-12.0,L={(u_8,\tilde{x}_2)}]{u_56}
\Vertex[x=\x, y=\y-15.0,L={(u_9,\tilde{x}_2)}]{u_66}
\Vertex[x=\x, y=\y-18.0,L={(u_{10},\tilde{x}_2)}]{u_76}

%

\Edge[style={in=135, out=225},labelstyle={xshift=4pt, yshift=4pt},label = a](u_11)(u_21)
\Edge[style={in=135, out=225},label = a](u_11)(u_31)

\Edge[style={in=135, out=225},labelstyle={xshift=4pt, yshift=4pt},label = a](u_12)(u_22)
\Edge[style={in=135, out=225},label = a](u_12)(u_32)

\Edge[style={in=135, out=225},labelstyle={xshift=4pt, yshift=4pt},label = a](u_13)(u_23)
\Edge[style={in=135, out=225},label = a](u_13)(u_33)

\Edge[style={in=135, out=225},labelstyle={xshift=4pt, yshift=4pt},label = a](u_14)(u_24)
\Edge[style={in=135, out=225},label = a](u_14)(u_34)

\Edge[style={in=135, out=225},labelstyle={xshift=4pt, yshift=4pt},label = a](u_15)(u_25)
\Edge[style={in=135, out=225},label = a](u_15)(u_35)

\Edge[style={in=135, out=225},labelstyle={xshift=4pt, yshift=4pt},label = a](u_16)(u_26)
\Edge[style={in=135, out=225},label = a](u_16)(u_36)

\Edge[style={in=135, out=240},labelstyle={xshift=0pt, yshift=4pt},label = f](u_51)(u_71)
\Edge[style={in=45, out=-45},label = f](u_61)(u_71)

\Edge[style={in=135, out=240},labelstyle={xshift=0pt, yshift=4pt},label = f](u_52)(u_72)
\Edge[style={in=45, out=-45},label = f](u_62)(u_72)

\Edge[style={in=135, out=240},labelstyle={xshift=0pt, yshift=4pt},label = f](u_53)(u_73)
\Edge[style={in=45, out=-45},label = f](u_63)(u_73)

\Edge[style={in=135, out=240},labelstyle={xshift=0pt, yshift=4pt},label = f](u_54)(u_74)
\Edge[style={in=45, out=-45},label = f](u_64)(u_74)

\Edge[style={in=135, out=240},labelstyle={xshift=0pt, yshift=4pt},label = f](u_55)(u_75)
\Edge[style={in=45, out=-45},label = f](u_65)(u_75)

\Edge[style={in=135, out=240},labelstyle={xshift=0pt, yshift=4pt},label = f](u_56)(u_76)
\Edge[style={in=45, out=-45},label = f](u_66)(u_76)

\Edge[style={in=135, out=45},labelstyle={xshift=1pt, yshift=0pt},label = d](u_12)(u_14)
\Edge[style={in=225, out=315},labelstyle={xshift=1pt, yshift=0pt},label = d](u_13)(u_15)

\Edge[style={in=135, out=45},labelstyle={xshift=1pt, yshift=0pt},label = d](u_22)(u_24)
\Edge[style={in=225, out=315},labelstyle={xshift=1pt, yshift=0pt},label = d](u_23)(u_25)

\Edge[style={in=135, out=45},labelstyle={xshift=1pt, yshift=0pt},label = d](u_32)(u_34)
\Edge[style={in=225, out=315},labelstyle={xshift=1pt, yshift=0pt},label = d](u_33)(u_35)

\Edge[style={in=135, out=45},labelstyle={xshift=1pt, yshift=0pt},label = d](u_42)(u_44)
\Edge[style={in=225, out=315},labelstyle={xshift=1pt, yshift=0pt},label = d](u_43)(u_45)

\Edge[style={in=135, out=45},labelstyle={xshift=1pt, yshift=0pt},label = d](u_52)(u_54)
\Edge[style={in=225, out=315},labelstyle={xshift=1pt, yshift=0pt},label = d](u_53)(u_55)

\Edge[style={in=135, out=45},labelstyle={xshift=1pt, yshift=0pt},label = d](u_62)(u_64)
\Edge[style={in=225, out=315},labelstyle={xshift=1pt, yshift=0pt},label = d](u_63)(u_65)

\Edge[style={in=135, out=45},labelstyle={xshift=1pt, yshift=0pt},label = d](u_72)(u_74)
\Edge[style={in=225, out=315},labelstyle={xshift=1pt, yshift=0pt},label = d](u_73)(u_75)

\Edge[style={in=180, out=225},labelstyle={xshift=1pt, yshift=0pt},label = c](u_21)(u_42)
\Edge[style={in=165, out=285},labelstyle={xshift=1pt, yshift=0pt},label = c](u_21)(u_43)
\Edge[style={in=165, out=285},labelstyle={xshift=1pt, yshift=0pt},label = c](u_31)(u_42)
\Edge[style={in=165, out=320},labelstyle={xshift=-15pt, yshift=-14pt},label = c](u_31)(u_43)

\Edge[style={in=90, out=30},labelstyle={xshift=10pt, yshift=-26pt},label = c](u_44)(u_56)
\Edge[style={in=120, out=0},labelstyle={xshift=-10pt, yshift=10pt},label = c](u_44)(u_66)
\Edge[style={in=120, out=330},labelstyle={xshift=-4pt, yshift=4pt},label = c](u_45)(u_56)
\Edge[style={in=120, out=315},labelstyle={xshift=0pt, yshift=4pt},label = c](u_45)(u_66)

\Edge[style={in=90, out=310},labelstyle={xshift=-24pt, yshift=+14pt},label = c](u_24)(u_46)
\Edge[style={in=90, out=270},labelstyle={xshift=-10pt, yshift=14pt},label = c](u_25)(u_46)
\Edge[style={in=105, out=300},labelstyle={xshift=-0pt, yshift=4pt},label = c](u_34)(u_46)
\Edge[style={in=90, out=300},labelstyle={xshift=-4pt, yshift=4pt},label = c](u_35)(u_46)

\Edge[style={in=115, out=330},labelstyle={xshift=24pt, yshift=-34pt},label = c](u_41)(u_52)
\Edge[style={in=120, out=315},labelstyle={xshift=0pt, yshift=0pt},label = c](u_41)(u_62)
\Edge[style={in=135, out=330},labelstyle={xshift=-6pt, yshift=6pt},label = c](u_41)(u_53)
\Edge[style={in=105, out=330},labelstyle={xshift=-14pt, yshift=11pt},label = c](u_41)(u_63)

\Edge[style={in=70, out=140,out looseness=1.3, in looseness=3.75},labelstyle={xshift=-20pt, yshift=0pt},label = b](u_21)(u_56)
\Edge[style={in=-45, out=250,out looseness=3.8, in looseness=1.5},labelstyle={xshift=-20pt, yshift=0pt},label = b](u_31)(u_66)


  \def\x{-4.5+10.0}
  \def\y{-17.0-16.5}
\draw[circle, -,dotted, very thick,rounded corners=8pt] (\x+4.8,\y-0.5)-- (\x+5.9+0.75,\y-0.5)-- (\x+5.9+0.75,\y+8.5-0.5)-- (\x+1.7,\y+11)-- (\x-6.8,\y+11)-- (\x-8,\y+12.5)-- (\x-8,\y+19.5)-- (\x-11,\y+19.5)-- (\x-11,\y+11.0)-- (\x-6.5+1,\y+8.0)-- (\x+2.3,\y+8.0)-- (\x+3.8,\y+6.5)-- (\x+3.8,\y-0.5)-- (\x+4.8,\y-0.5);

  \def\x{-5.2}
  \def\y{-12.0}
\draw[circle, -,dashed, very thick,rounded corners=8pt] (\x+0.2,\y+2)--(\x+17.4,\y+2) --(\x+17.9,\y+1.5) -- (\x+17.9,\y-25.5)-- (\x+17.4,\y-26) -- (\x-0.3,\y-26) -- (\x-0.8,\y-25.5) -- (\x-0.8,\y+1.5) -- (\x-0.3,\y+2)--(\x+0.1,\y+2);

  \def\x{-18.0}
  \def\y{-5.25}
\draw[circle, -,dashed, very thick,rounded corners=8pt] (\x-0.5,\y+1)--(\x+1.8,\y+2)--(\x+2.4,\y+2) --(\x+2.9,\y+1.5) -- (\x+2.9,\y-1)-- (\x+2.4,\y-1.5) -- (\x+1.9,\y-1.5) -- (\x-0.7,\y-0.5)--(\x-0.7,\y+0.8)--(\x-0.5,\y+1);

  \def\x{-14.0}
  \def\y{-5.3}
\draw[circle, -,dashed, very thick,rounded corners=8pt] (\x+0.2,\y+2)--(\x+2.4,\y+2) --(\x+2.9,\y+1.5) -- (\x+2.9,\y+1-2)-- (\x+2.4,\y+0.5-2) -- (\x-0.3,\y+0.5-2) -- (\x-0.8,\y+1-2) -- (\x-0.8,\y+1.5) -- (\x-0.3,\y+2)--(\x+0.1,\y+2);

  \def\x{-10.0}
  \def\y{-5.3}
\draw[circle, -,dashed, very thick,rounded corners=8pt] (\x+0.2,\y+2)--(\x+4.4-2,\y+2-1) --(\x+4.9-2,\y+1.5-1) -- (\x+4.9-2,\y+1-1)-- (\x+4.4-2,\y+0.5-1) -- (\x-0.3,\y+0.5-2) -- (\x-0.8,\y+1-2) -- (\x-0.8,\y+1.5-2) -- (\x-0.8,\y+1.5) -- (\x-0.3,\y+2)--(\x+0.1,\y+2);

\end{tikzpicture}
}
\end{center}
\caption{Decomposition of $G\cong G/Y\boxbackslash G/X_1/X_2$. The set $Z$ from the proof of Theorem~\ref{theorem_4} and the graph isomorphic to $G$ induced by $Z$ in $G/Y\boxtimes G/X_1/X_2$ is indicated within the dotted region (apart from the arcs with label pair $b$ which are partially outside this region).}
  \label{Example1}
\end{figure}
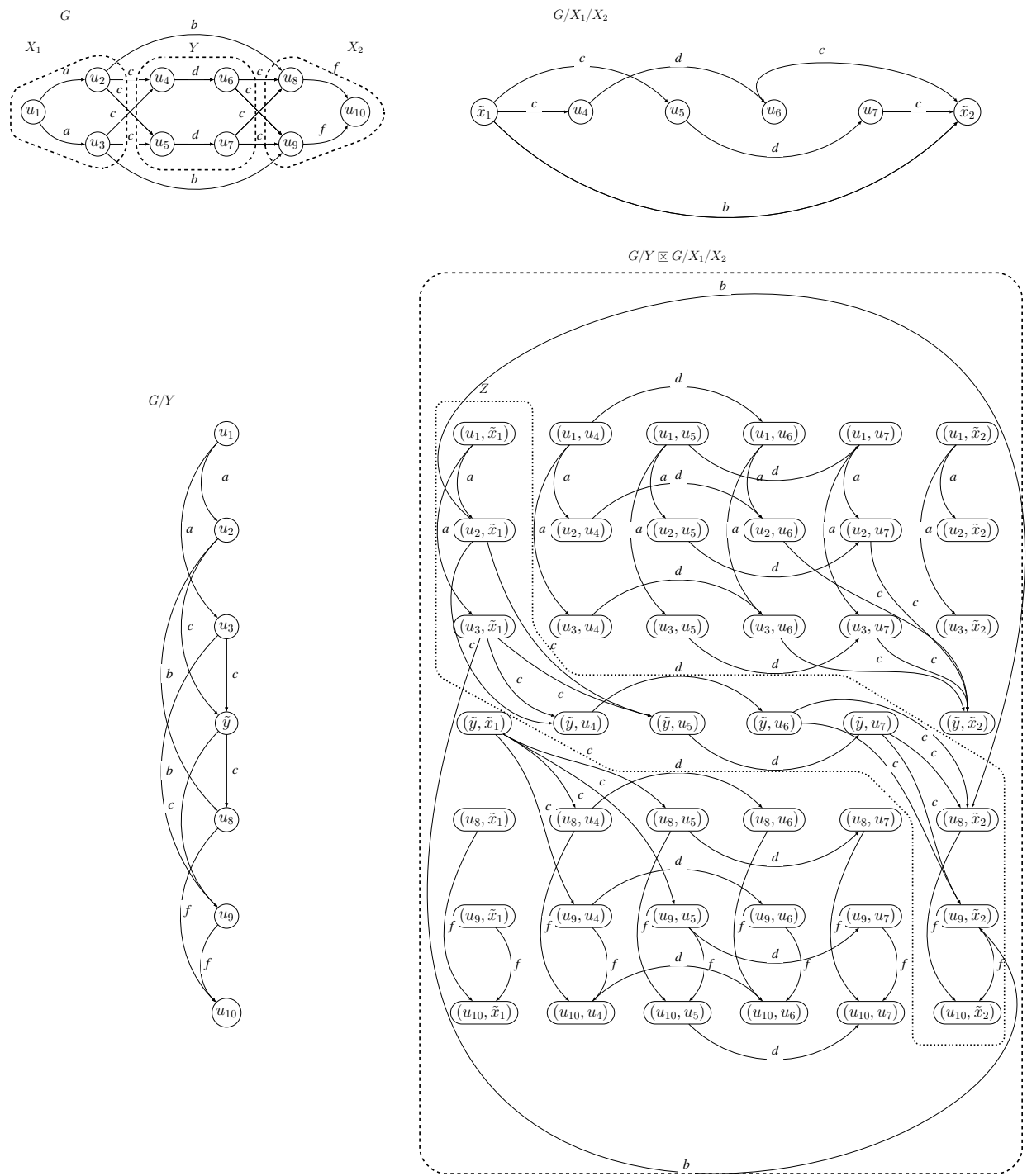
Before we can prove Theorem~\ref{theorem_3} and Theorem~\ref{theorem_4}, we state and prove in Lemma~\ref{lemma1} that a (not necessarily complete) bipartite graph $B(X,Y)$ consisting solely of complete bipartite subgraphs $B(X_i,Y_i), i=1,\ldots n,$ can be decomposed in such a manner that $B(X,Y)\cong  B(X,Y)/Y\boxbackslash B(X,Y)/X$, where $X_i\subseteq X, Y_i\subseteq Y$, all arcs of $B(X_i,Y_i)$ have the same label pair, all $[X_i,Y_i]$ have no backward arcs or all $[X_i,Y_i]$ have no forward arcs, and any pair of subgraphs $B(X_i,Y_i)$ and $B(X_j,Y_j), i\neq j,$ have different label pairs.
In Figure~\ref{BiPartiteExample1}, we give a simple example of the decomposition of a bipartite graph where all arcs have the same label pair.
Because all label pairs are identical, we have omitted these label pairs.
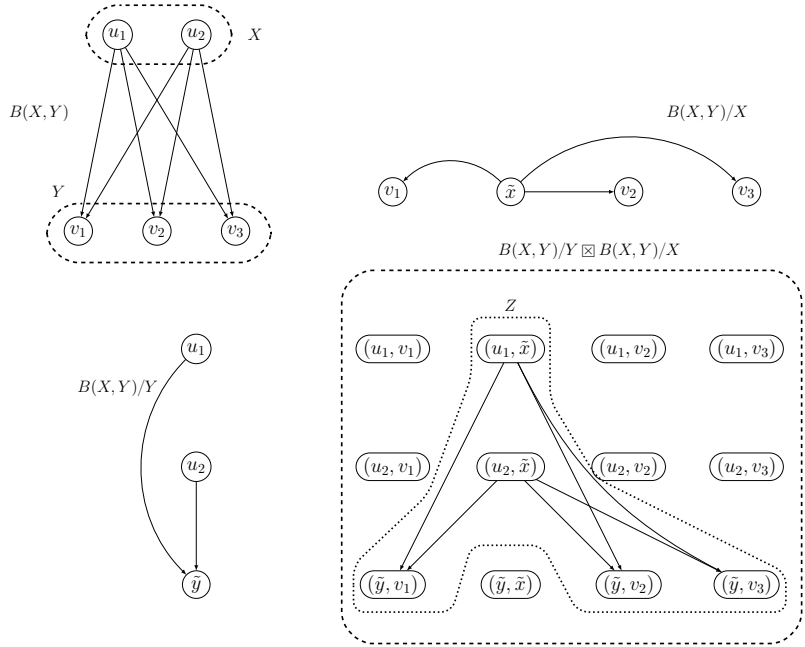
\begin{figure}[H]
\begin{center}
\resizebox{0.65\textwidth}{!}{
\begin{tikzpicture}[->,>=latex,shorten >=0pt,auto,node distance=2.5cm,
  main node/.style={circle,fill=blue!10,draw, font=\sffamily\Large\bfseries}]
  \tikzset{VertexStyle/.append style={
  font=\itshape\large, shape = circle,inner sep = 2pt, outer sep = 0pt,minimum size = 20 pt,draw}}
  \tikzset{EdgeStyle/.append style={thin}}
  \tikzset{LabelStyle/.append style={font = \itshape}}
  \SetVertexMath
  \def\x{0.0}
  \def\y{0.0}
\node at (\x-12.0,\y-6) {$B(X,Y)$};
\node at (\x-6.5,\y-4) {$X$};
\node at (\x-11.5,\y-8) {$Y$};
\node at (\x+5.0,\y-6) {$B(X,Y)/X$};
\node at (\x-10.0,\y-13) {$B(X,Y)/Y$};
\node at (\x+2.0,\y-9.5) {$B(X,Y)/Y\boxtimes B(X,Y)/X$};
\node at (\x-0.0,\y-10.9) {$Z$};
  \def\x{-10.0}
  \def\y{-6.0}
  \Vertex[x=\x-0, y=\y+2.0]{u_1}
  \Vertex[x=\x+2, y=\y+2.0]{u_2}
  \def\x{-9.0}
  \Vertex[x=\x-2, y=\y-3]{v_1}
  \Vertex[x=\x+0, y=\y-3]{v_2}
  \Vertex[x=\x+2, y=\y-3]{v_3}

  \Edge(u_1)(v_1) 
  \Edge(u_1)(v_2) 
  \Edge(u_1)(v_3) 
  \Edge(u_2)(v_1) 
  \Edge(u_2)(v_2) 
  \Edge(u_2)(v_3) 
  
  \def\x{-3.0}
  \def\y{-7.0}
  \Vertex[x=\x-0, y=\y-1]{v_1}
  \Vertex[x=\x+3, y=\y-1,L={\tilde{x}}]{u_1}
  \Vertex[x=\x+6, y=\y-1]{v_2}
  \Vertex[x=\x+9, y=\y-1]{v_3}

  \Edge[style={bend left=-45,min distance=1cm}](u_1)(v_1) 
  \Edge(u_1)(v_2) 
  \Edge[style={bend right=-45,min distance=1cm}](u_1)(v_3) 
  
  \def\x{-8.0}
  \def\y{-14.0}
  \Vertex[x=\x-0, y=\y+2.0]{u_1}
  \Vertex[x=\x+0, y=\y-1.0]{u_2}
  \Vertex[x=\x+0, y=\y-4,L={\tilde{y}}]{v_1}

  \Edge[style={bend left=-45,min distance=1cm}](u_1)(v_1) 
  \Edge(u_2)(v_1)

  \tikzset{VertexStyle/.append style={
  font=\itshape\large, shape = rounded rectangle, inner sep = 2pt, outer sep = 0pt,minimum size = 20 pt,draw}}

  \def\x{-3.0}
  \def\y{-12.0}
  \Vertex[x=\x, y=\y+0.0,L={(u_1,v_1)}]{u_11}
  \Vertex[x=\x+3, y=\y+0.0,L={(u_1,\tilde{x})}]{u_12}
  \Vertex[x=\x+6, y=\y+0.0,L={(u_1,v_2)}]{u_13}
  \Vertex[x=\x+9, y=\y+0.0,L={(u_1,v_3)}]{u_14}

  \def\x{-3.0}
  \def\y{-15.0}
  \Vertex[x=\x, y=\y+0.0,L={(u_2,v_1)}]{u_21}
  \Vertex[x=\x+3, y=\y+0.0,L={(u_2,\tilde{x})}]{u_22}
  \Vertex[x=\x+6, y=\y+0.0,L={(u_2,v_2)}]{u_23}
  \Vertex[x=\x+9, y=\y+0.0,L={(u_2,v_3)}]{u_24}

  \def\x{-3.0}
  \def\y{-18.0}
  \Vertex[x=\x, y=\y+0.0,L={(\tilde{y},v_1)}]{u_31}
  \Vertex[x=\x+3, y=\y+0.0,L={(\tilde{y},\tilde{x})}]{u_32}
  \Vertex[x=\x+6, y=\y+0.0,L={(\tilde{y},v_2)}]{u_33}
  \Vertex[x=\x+9, y=\y+0.0,L={(\tilde{y},v_3)}]{u_34}

  \Edge(u_12)(u_31) 
  \Edge(u_22)(u_31) 
  \Edge(u_12)(u_33) 
  \Edge(u_22)(u_33) 
  \Edge[style={bend left=-18,min distance=1cm}](u_12)(u_34) 
  \Edge(u_22)(u_34) 

  \def\x{-3.5}
  \def\y{-17.0}
\draw[circle, -,dotted, very thick,rounded corners=8pt] (\x-0.5,\y-1.0)-- (\x-0.5,\y-0.5)-- (\x+1.4,\y+1.3)-- (\x+2.5,\y+4.3)-- (\x+2.5,\y+5.8)-- (\x+4.5,\y+5.8)-- (\x+4.5,\y+3.7)-- (\x+5.7,\y+1.7)--(\x+10.5,\y-0.7)--(\x+10.5,\y-1.7)--(\x+5.2,\y-1.7)--(\x+4.3,\y+0.0)--(\x+2.5,\y+0.0)-- (\x+2.0,\y-1.7)-- (\x-0.5,\y-1.7)  --  (\x-0.5,\y-1.0);
  \def\x{-3.5}
  \def\y{-12.0}
\draw[circle, -,dashed, very thick,rounded corners=8pt] (\x+0.2,\y+2)--(\x+10.4,\y+2) --(\x+10.9,\y+1.5) -- (\x+10.9,\y-7)-- (\x+10.4,\y-7.5) -- (\x-0.3,\y-7.5) -- (\x-0.8,\y-7) -- (\x-0.8,\y+1.5) -- (\x-0.3,\y+2)--(\x+0.1,\y+2);

  \def\x{-10.0}
  \def\y{-5.25}
\draw[circle, -,dashed, very thick,rounded corners=8pt] (\x+0.2,\y+2)--(\x+2.4,\y+2) --(\x+2.9,\y+1.5) -- (\x+2.9,\y+1)-- (\x+2.4,\y+0.5) -- (\x-0.3,\y+0.5) -- (\x-0.8,\y+1) -- (\x-0.8,\y+1.5) -- (\x-0.3,\y+2)--(\x+0.1,\y+2);

  \def\x{-11.0}
  \def\y{-10.3}
\draw[circle, -,dashed, very thick,rounded corners=8pt] (\x+0.2,\y+2)--(\x+4.4,\y+2) --(\x+4.9,\y+1.5) -- (\x+4.9,\y+1)-- (\x+4.4,\y+0.5) -- (\x-0.3,\y+0.5) -- (\x-0.8,\y+1) -- (\x-0.8,\y+1.5) -- (\x-0.3,\y+2)--(\x+0.1,\y+2);

\end{tikzpicture}
}
\end{center}
\caption{Decomposition of $B(X,Y)\cong B(X,Y)/Y\boxbackslash B(X,Y)/X$. The set $Z$ from the proof of Lemma~\ref{lemma1} and the graph isomorphic to $B(X,Y)$ induced by $Z$ in $B(X,Y)/X\boxtimes B(X,Y)/Y$ is indicated within the dotted region. Because all label pairs are identical, we have omitted these label pairs.}
  \label{BiPartiteExample1}
\end{figure}
The decomposition given in Lemma~\ref{lemma1} is restricted to a clean bipartite graph.
Note that we allow parallel arcs with different label pairs in $B(X,Y)$.
Furthermore, note that $B(X,Y)$ is not necessarily weakly connected.

\begin{lemma}\label{lemma1}
Let $B(X,Y)$ be a clean bipartite graph.
Then $B(X,Y)\cong B(X,Y)/Y\boxbackslash B(X,Y)/X$.
\end{lemma}
\begin{proof}
It suffices to define a mapping $\phi: V(B(X,Y))\rightarrow V(B(X,Y)/Y\boxbackslash B(X,Y)/X)$ and to prove that $\phi$ is an isomorphism from $B(X,Y)$ to $B(X,Y)/Y\boxbackslash B(X,Y)/X$.
Let $\tilde{x}$ and $\tilde{y}$ be the new vertices replacing the sets $X$ and $Y$ when defining $B(X,Y)/X$ and $B(X,Y)/Y$, respectively. 
Consider the mapping $\phi: V(B(X,Y))\rightarrow V(B(X,Y)/Y\boxbackslash B(X,Y)/X)$ defined by  $\phi(u)=(u,\tilde{x})$ for all $u\in X$, and $\phi(v)=(\tilde{y},v)$ for all $v\in Y$.
Then $\phi$ is obviously a bijection if  $V(B(X,Y)/Y \boxbackslash B(X,Y)/X)=Z$, where $Z$ is defined as $Z=\{(u,\tilde{x}) \mid u\in X\}\cup \{(\tilde{y},v)\mid v\in Y\}$. We are going to show this later by arguing that all the other vertices of $B(X,Y)/Y\,\Box\, B(X,Y)/X$ will disappear from $B(X,Y)/Y\boxtimes B(X,Y)/X$. 
But first we are going to prove the following claim. 
\begin{claim}\label{claim2}
The subgraph of $B(X,Y)/Y\boxtimes B(X,Y)/X$ induced by $Z$ is isomorphic to $B(X,Y)$.
\end{claim}
\begin{proof}
Obviously, $\phi$ is a bijection from $V(B(X,Y))$ to $Z$.
It remains to show that this bijection preserves the arcs and their label pairs.
Let $X=\{u_1,\ldots,u_m\}, Y=\{v_1,\ldots,v_n\}$ be the disjoint vertex sets of a clean bipartite graph $B(X,Y)$.
Let $L=\{\lambda_1, \ldots, \lambda_x\}$ be the set of label pairs belonging to $B(X,Y)$.
Let all arcs of $A(B(X,Y))$ with label pair $\lambda_i$ arc-induce the clean bipartite subgraph $B(X_i,Y_i)$.
Then, $X=\overundercup{i=1}{x}X_i$ and $Y=\overundercup{i=1}{x}Y_i$.
Note that $X_i\cap X_j$ and $Y_i\cap Y_j, i\neq j$, are not necessarily empty sets and note that $B(X_i,Y_i)$ is complete.
Let $[X,Y]$ have no backward arcs. Hence, $[X_i,Y_i],i=1\ldots x,$ have no backward arcs.
Because, $X_i\subseteq X$ and $Y_i\subseteq Y$, and $\tilde{x}$ and $\tilde{y}$ are the new vertices replacing the sets $X$ and $Y$ when defining $B(X,Y)/X$ and $B(X,Y)/Y$, respectively, we have that $X_i$ and $Y_i$ (when defining $B(X_i,Y_i)/X_i$ and $B(X_i,Y_i)/Y_i$) are replaced by $\tilde{x}$ and $\tilde{y}$, respectively.

Now, we will prove that the subgraph of $B(X_i,Y_i)/Y_i\boxtimes B(X_i,Y_i)/X_i$ induced by $Z_i=\{(u,\tilde{x})\mid u\in X_i \cup \{\tilde{y},v)\mid v\in Y_i\}\subseteq Z$ is isomorphic to $B(X_i,Y_i)$.
Obviously, the mapping $\phi$ restricted to $V(B(X_i,Y_i))$ is a bijection from $V(B(X_i,Y_i))$ to $Z_i$.
It remains to show that this bijection preserves the arcs and their label pairs.
Let $X_i=\{u_{i_1},\ldots,u_{i_k}\}\subseteq X, Y=\{v_{i_1},\ldots,v_{i_l}\}\subseteq Y$ be the disjoint vertex sets of $B(X_i,Y_i)$.

$B(X_i,Y_i)$ is a clean bipartite graph, $B(X_i,Y_i)$ has the arc set $A_i=\{a\mid \mu(a)=(u_{i_s},v_{j_t}), a\in [X_i,Y_i]\}$ for $1\leq s\leq k$ and $1\leq t \leq l$, and $|A_i|=k\cdot l$.
Any two arcs $b$ with $\mu(b)=(u_{i_s},\tilde{y})$ in $B(X_i,Y_i)/Y_i$ and $c$ with $\mu(c)=(\tilde{x},v_{j_t})$ in $B(X_i,Y_i)/X_i$ are synchronising arcs, because $\lambda(b)=\lambda(c)$. 
Due to the VRSP, the arcs $b$ in $B(X_i,Y_i)/Y_i$ and $c$ in $B(X_i,Y_i)/X_i$ correspond to an arc $d$ with $\mu(d)=((u_{i_s},\tilde{x}),(\tilde{y},v_{j_t}))=(\phi(u_{i_s}),\phi(v_{j_t}))$ in $B(X_i,Y_i)/Y_i\boxtimes B(X_i,Y_i)/X_i$ with $\lambda(b)=\lambda(d)$.
Because the arc set $A_i=A(B(X_i,Y_i)/Y_i)=\{b\mid \mu(b)=(u_{i_s},\tilde{y})\}$ has cardinality $k$, the arc set $A(B(X_i,Y_i)/X_i)=\{c\mid \mu(c)=(\tilde{x},v_{j_t})\}$ has cardinality $l$ and all arcs of $A(B(X_i,Y_i)/Y_i)$ and $A(B(X_i,Y_i)/X_i)$ have identical label pairs, it follows that the arc set $A'_i=\{d\mid \mu(d)=((u_{i_s},\tilde{x}),(\tilde{y},v_{j_t}))=(\phi(u_{i_s}),\phi(v_{j_t})), 1\leq s\leq k, 1\leq t\leq l\}\subseteq A(B(X_i,Y_i)/Y_i\boxtimes B(X_i,Y_i)/X_i)$ has cardinality $k\cdot l$.
Furthermore, $\phi$ restricted to $V(B(X_i,Y_i))$ maps vertices $u_{i_s}$ and $v_{j_t}$ onto vertices $(u_{i_s},\tilde{x})$ and $(\tilde{y},v_{j_t})$, respectively, and therefore we have an arc $a$ with $\mu(a)=(u_{i_s},v_{j_t})$ in $B(X_i,Y_i)$ which corresponds to the arc $d$ with $\mu(d)=((u_{i_s},\tilde{x}),(\tilde{y},v_{j_t}))$ in $B(X_i,Y_i)/Y_i\boxtimes B(X_i,Y_i)/X_i$, with $\lambda(a)=\lambda(d)$. 
Together with $|A_i|=|A'_i|$, we have the one-to-one relationship between the arc $d$ in $B(X_i,Y_i)/Y_i\boxtimes B(X_i,Y_i)/X_i$ and the arc $a$ in $B(X_i,Y_i)$.
Therefore, because there are no other vertices in $Z_i$ than $(u_{i_s},\tilde{x})$ and $(\tilde{y},v_{j_t})$ and there are no other vertices in $B(X_i,Y_i)$ then $(u_{i_s},v_{j_t})$, the subgraph of $B(X_i,Y_i)/Y_i\boxtimes B(X_i,Y_i)/X_i$ arc-induced by the arcs of $B(X_i,Y_i)/Y_i\boxtimes B(X_i,Y_i)/X_i$ with label pair $\lambda_i$ is isomorphic to $B(X_i,Y_i)$.
This is valid for all $B(X_i,Y_i)$ because $ \lambda_i\neq \lambda_j, i\neq j,$ $\overundercup{i=1}{x}X_i=X$, $\overundercup{i=1}{x}Y_i=Y$ and $\overundercup{i=1}{x}Z_i=Z$.
Therefore,  we have that the subgraph of $B(X,Y)/Y\boxtimes  B(X,Y)/X$ induced by $Z$ is isomorphic to $B(X,Y)$.
This completes the proof of Claim~\ref{claim2}. 
\end{proof}

It remains to show that $\phi$ is a bijection from $V(B(X,Y))$ to $Z'=V(B(X,Y)/Y\boxbackslash B(X,Y)$ $/X)$.
Now, we have $Z'\subseteq V(B(X,Y)/Y\boxtimes B(X,Y)/X)=\{(u_i,v_j)\}\cup \{(u_i,\tilde{x})\}\cup\{(\tilde{y},v_j)\}\cup\{(\tilde{y},\tilde{x})\}$.
The arcs $b$ with $\mu(b)=(u_i,\tilde{x})$ in $B(X,Y)/Y$ and $c$ with $\mu(c)=(\tilde{y},v_j)$ in $B(X,Y)/X$ are synchronising arcs.
Therefore, the only vertices that are the tail of an arc in $B(X,Y)/Y\boxtimes B(X,Y)/X$ are $(u_i,\tilde{x})$ and the only vertices that are the head of an arc in $B(X,Y)/Y\boxtimes B(X,Y)/X$ are $(\tilde{y},v_j)$.
Next, the vertices $u_i$ in $B(X,Y)/Y$ and the vertex $\tilde{x}$ in $B(X,Y)/X$ have $level\,0$.
All other vertices in $B(X,Y)/Y$ and $B(X,Y)/X$ have $level\,1$.
Therefore, the only vertices in $B(X,Y)/Y\,\Box\, B(X,Y)$ $/X$ with $level\,0$ are the vertices $(u_i,\tilde{x})$.
It follows that the vertices $(u_i,v_j)$ and $(\tilde{y},\tilde{x})$ are removed from $V(B(X,Y)/Y\boxtimes B(X,Y)/X)$ because $level((u_i,v_j))>0$ in $B(X,Y)/Y\,\Box\, B(X,Y)/X$ but $level((u_i,v_j))=0$ in $B(X,Y)/Y\boxtimes B(X,Y)/X$ and $level((\tilde{y},\tilde{x}))$ $>0$ in $B(X,Y)/Y\,\Box\, B(X,Y)/$ $X$ but $level((\tilde{y}$, $\tilde{x}))=0$ in $B(X,Y)/Y\boxtimes B(X,Y)/X$.
Therefore, it follows that $Z'=\{(u_i,\tilde{x})\}\cup\{(\tilde{y},v_j)\}=Z$, for $1\leq i\leq m$ and $1\leq j \leq n$. 
Hence, $\phi$ is a bijection from $V(B(X,Y))$ to $Z$ preserving the arcs and their label pairs and therefore $B(X,Y)\cong B(X,Y)/Y\boxbackslash B(X,Y)/X$.
With similar arguments, it follows that $B(X,Y)\cong B(X,Y)/Y\boxbackslash B(X,Y)/X$ if $[X,Y]$ contains no forward arcs.
This completes the proof of Lemma~\ref{lemma1}.
\end{proof}
In Figure~\ref{BiPartiteExample2}, we give a bipartite graph where all arcs have identical label pairs which is not clean. 
For the arc $a$ with $\mu(a)=((u_1,\tilde{x}),(\tilde{y},v_1))$ in $B(X,Y)/Y\boxtimes B(X,Y)/X$ there is no arc $b$ with $\mu(b)=(u_1,v_1)$ in $B(X,Y)$.
Hence, $B(X,Y)\ncong B(X,Y)/Y\boxbackslash B(X,Y)/X$.
Therefore, we cannot relax the condition on the completeness of the bipartite graph without violating the conclusion of Lemma~\ref{lemma1}.

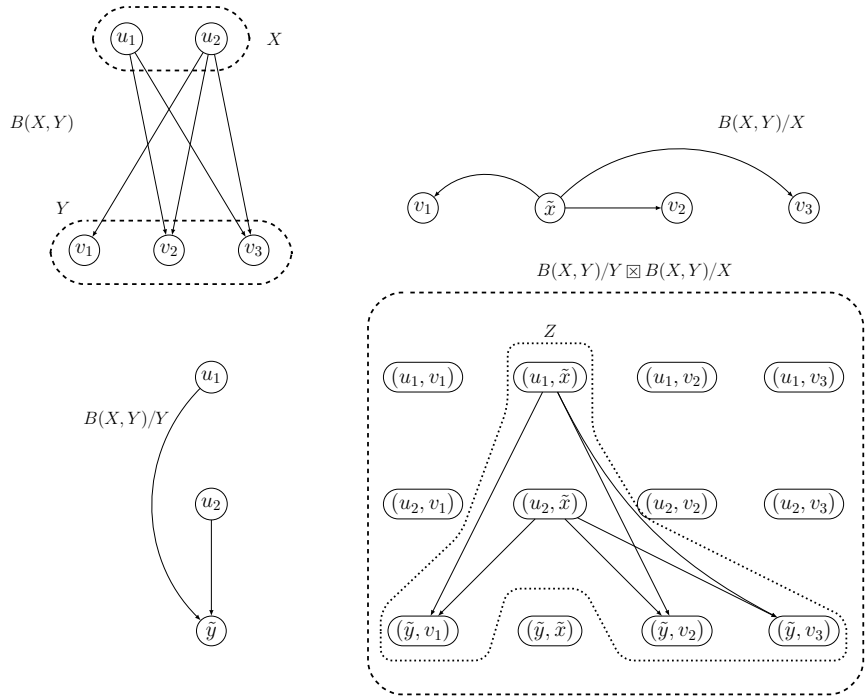
\begin{figure}[H]
\begin{center}
\resizebox{0.7\textwidth}{!}{
\begin{tikzpicture}[->,>=latex,shorten >=0pt,auto,node distance=2.5cm,
  main node/.style={circle,fill=blue!10,draw, font=\sffamily\Large\bfseries}]
  \tikzset{VertexStyle/.append style={
  font=\itshape\large, shape = circle,inner sep = 2pt, outer sep = 0pt,minimum size = 20 pt,draw}}
  \tikzset{EdgeStyle/.append style={thin}}
  \tikzset{LabelStyle/.append style={font = \itshape}}
  \SetVertexMath
  \def\x{0.0}
  \def\y{0.0}
\node at (\x-12.0,\y-6) {$B(X,Y)$};
\node at (\x-6.5,\y-4) {$X$};
\node at (\x-11.5,\y-8) {$Y$};
\node at (\x+5.0,\y-6) {$B(X,Y)/X$};
\node at (\x-10.0,\y-13) {$B(X,Y)/Y$};
\node at (\x+2.0,\y-9.5) {$B(X,Y)/Y\boxtimes B(X,Y)/X$};
\node at (\x-0.0,\y-10.9) {$Z$};
  \def\x{-10.0}
  \def\y{-6.0}
  \Vertex[x=\x-0, y=\y+2.0]{u_1}
  \Vertex[x=\x+2, y=\y+2.0]{u_2}
  \def\x{-9.0}
  \Vertex[x=\x-2, y=\y-3]{v_1}
  \Vertex[x=\x+0, y=\y-3]{v_2}
  \Vertex[x=\x+2, y=\y-3]{v_3}

  \Edge(u_1)(v_2) 
  \Edge(u_1)(v_3) 
  \Edge(u_2)(v_1) 
  \Edge(u_2)(v_2) 
  \Edge(u_2)(v_3) 
  
  \def\x{-3.0}
  \def\y{-7.0}
  \Vertex[x=\x-0, y=\y-1]{v_1}
  \Vertex[x=\x+3, y=\y-1,L={\tilde{x}}]{u_1}
  \Vertex[x=\x+6, y=\y-1]{v_2}
  \Vertex[x=\x+9, y=\y-1]{v_3}

  \Edge[style={bend left=-45,min distance=1cm}](u_1)(v_1) 
  \Edge(u_1)(v_2) 
  \Edge[style={bend right=-45,min distance=1cm}](u_1)(v_3) 
  
  \def\x{-8.0}
  \def\y{-14.0}
  \Vertex[x=\x-0, y=\y+2.0]{u_1}
  \Vertex[x=\x+0, y=\y-1.0]{u_2}
  \Vertex[x=\x+0, y=\y-4,L={\tilde{y}}]{v_1}

  \Edge[style={bend left=-45,min distance=1cm}](u_1)(v_1) 
  \Edge(u_2)(v_1) 
  
  \tikzset{VertexStyle/.append style={
  font=\itshape\large, shape = rounded rectangle, inner sep = 2pt, outer sep = 0pt,minimum size = 20 pt,draw}}

  \def\x{-3.0}
  \def\y{-12.0}
  \Vertex[x=\x, y=\y+0.0,L={(u_1,v_1)}]{u_11}
  \Vertex[x=\x+3, y=\y+0.0,L={(u_1,\tilde{x})}]{u_12}
  \Vertex[x=\x+6, y=\y+0.0,L={(u_1,v_2)}]{u_13}
  \Vertex[x=\x+9, y=\y+0.0,L={(u_1,v_3)}]{u_14}

  \def\x{-3.0}
  \def\y{-15.0}
  \Vertex[x=\x, y=\y+0.0,L={(u_2,v_1)}]{u_21}
  \Vertex[x=\x+3, y=\y+0.0,L={(u_2,\tilde{x})}]{u_22}
  \Vertex[x=\x+6, y=\y+0.0,L={(u_2,v_2)}]{u_23}
  \Vertex[x=\x+9, y=\y+0.0,L={(u_2,v_3)}]{u_24}

  \def\x{-3.0}
  \def\y{-18.0}
  \Vertex[x=\x, y=\y+0.0,L={(\tilde{y},v_1)}]{u_31}
  \Vertex[x=\x+3, y=\y+0.0,L={(\tilde{y},\tilde{x})}]{u_32}
  \Vertex[x=\x+6, y=\y+0.0,L={(\tilde{y},v_2)}]{u_33}
  \Vertex[x=\x+9, y=\y+0.0,L={(\tilde{y},v_3)}]{u_34}

  \Edge(u_12)(u_31) 
  \Edge(u_22)(u_31) 
  \Edge(u_12)(u_33) 
  \Edge(u_22)(u_33) 
  \Edge[style={bend left=-18,min distance=1cm}](u_12)(u_34) 
  \Edge(u_22)(u_34) 
  \def\x{-3.5}
  \def\y{-17.0}
\draw[circle, -,dotted, very thick,rounded corners=8pt] (\x-0.5,\y-1.0)-- (\x-0.5,\y-0.5)-- (\x+1.4,\y+1.3)-- (\x+2.5,\y+4.3)-- (\x+2.5,\y+5.8)-- (\x+4.5,\y+5.8)-- (\x+4.5,\y+3.7)-- (\x+5.7,\y+1.7)--(\x+10.5,\y-0.7)--(\x+10.5,\y-1.7)--(\x+5.2,\y-1.7)--(\x+4.3,\y+0.0)--(\x+2.5,\y+0.0)-- (\x+2.0,\y-1.7)-- (\x-0.5,\y-1.7)  --  (\x-0.5,\y-1.0);
  \def\x{-3.5}
  \def\y{-12.0}
\draw[circle, -,dashed, very thick,rounded corners=8pt] (\x+0.2,\y+2)--(\x+10.4,\y+2) --(\x+10.9,\y+1.5) -- (\x+10.9,\y-7)-- (\x+10.4,\y-7.5) -- (\x-0.3,\y-7.5) -- (\x-0.8,\y-7) -- (\x-0.8,\y+1.5) -- (\x-0.3,\y+2)--(\x+0.1,\y+2);

  \def\x{-10.0}
  \def\y{-5.25}
\draw[circle, -,dashed, very thick,rounded corners=8pt] (\x+0.2,\y+2)--(\x+2.4,\y+2) --(\x+2.9,\y+1.5) -- (\x+2.9,\y+1)-- (\x+2.4,\y+0.5) -- (\x-0.3,\y+0.5) -- (\x-0.8,\y+1) -- (\x-0.8,\y+1.5) -- (\x-0.3,\y+2)--(\x+0.1,\y+2);

  \def\x{-11.0}
  \def\y{-10.3}
\draw[circle, -,dashed, very thick,rounded corners=8pt] (\x+0.2,\y+2)--(\x+4.4,\y+2) --(\x+4.9,\y+1.5) -- (\x+4.9,\y+1)-- (\x+4.4,\y+0.5) -- (\x-0.3,\y+0.5) -- (\x-0.8,\y+1) -- (\x-0.8,\y+1.5) -- (\x-0.3,\y+2)--(\x+0.1,\y+2);

\end{tikzpicture}
}
\end{center}
\caption{Decomposition of $B(X,Y)$ for which $B(X,Y)\ncong B(X,Y)/Y\boxbackslash B(X,Y)/X$. Because all label pairs are identical, we have omitted these label pairs.}
  \label{BiPartiteExample2}
\end{figure}

Using Lemma~\ref{lemma1}, we relax Theorem~\ref{theorem_1} and Theorem~\ref{theorem_2} leading to Theorem~\ref{theorem_3} and Theorem~\ref{theorem_4}, respectively.
We assume that the graphs we want to decompose are connected; if not, we can apply our decomposition results to the components separately. 
In Figure~\ref{FirstCounterDecomposition}, we show the decomposition of a graph $G$ that contains a complete bipartite subgraph $B(Z_1,Z_2)$ where all arcs of $B(Z_1,Z_2)$ have the label pair $s$.

\begin{figure}[H]
\begin{center}
\resizebox{1\textwidth}{!}{
\begin{tikzpicture}[->,>=latex,shorten >=0pt,auto,node distance=2.5cm,
  main node/.style={circle,fill=blue!10,draw, font=\sffamily\Large\bfseries}]
  \tikzset{VertexStyle/.append style={
  font=\itshape\large,shape = circle,inner sep = 0pt, outer sep = 0pt,minimum size = 20 pt,draw}}
  \tikzset{EdgeStyle/.append style={thin}}
  \tikzset{LabelStyle/.append style={font = \itshape}}
  \SetVertexMath
  \def\x{-6.0}
  \def\y{0.0}
\node at (\x-1.0,\y) {$G$};
\node at (\x+1,\y+2.0) {$X$};
\node at (\x+5,\y+2.0) {$Y$};
\node at (\x+2,\y+0.0) {$Z_1$};
\node at (\x+4,\y+0.0) {$Z_2$};
  \def\x{4.0}
\node at (\x-0.5,\y+1.5) {$G/X$};
\node at (\x-9.5,\y-3.5) {$G/Y$};
\node at (\x-2+1,\y-3.0) {$G/Y\boxbackslash G/X$};
\node at (\x-1+1,\y-6.75) {$B(Z_1,Z_2)/Z_2\boxbackslash B(Z_1,Z_2)/Z_1$};
  \def\x{-6.0}
  \def\y{0.0}
  \Vertex[x=\x+0, y=\y+0.0]{u_1}
  \Vertex[x=\x+2, y=\y+1.0]{u_2}
  \Vertex[x=\x+4, y=\y+1.0]{u_3}
  \Vertex[x=\x+6, y=\y+0.0]{u_4}
  \Vertex[x=\x+2, y=\y-1.0]{u_5}
  \Vertex[x=\x+4, y=\y-1.0]{u_6}
  \Edge[label = a](u_1)(u_2) 
  \Edge[label = s](u_2)(u_3) 
  \Edge[label = c](u_3)(u_4) 
  \Edge[label = b](u_1)(u_5) 
  \Edge[label = s](u_5)(u_6) 
  \Edge[label = s](u_2)(u_6) 
  \Edge[label = s](u_5)(u_3) 
  \Edge[label = d](u_6)(u_4) 
  \Edge(u_2)(u_6) 
  \Edge(u_5)(u_3) 

  \def\x{3.0}
  \def\y{0.0}
  \Vertex[x=\x+0, y=\y+0.0,L={\tilde{x}}]{s_1}
  \Vertex[x=\x+2, y=\y+1.0]{u_3}
  \Vertex[x=\x+4, y=\y+0.0]{u_4}
  \Vertex[x=\x+2, y=\y-1.0]{u_6}
  \Edge[label = s](s_1)(u_3) 
  \Edge[label = c](u_3)(u_4) 
  \Edge[label = s](s_1)(u_6) 
  \Edge[label = d](u_6)(u_4) 

  \def\x{-6.0}
  \def\y{-5.0}
  \Vertex[x=\x+0, y=\y+0.0]{u_1}
  \Vertex[x=\x+2, y=\y+1.0]{u_2}
  \Vertex[x=\x+2, y=\y-1.0]{u_5}
  \Vertex[x=\x+4, y=\y+0.0,L={\tilde{y}}]{t}
  \Edge[label = a](u_1)(u_2) 
  \Edge[label = s](u_2)(t) 
  \Edge[label = b](u_1)(u_5) 
  \Edge[label = s](u_5)(t) 

  \Edge(u_1)(u_2) 
  \Edge(u_2)(t) 

\tikzset{VertexStyle/.append style={
  font=\itshape\large,shape = rounded rectangle,inner sep = 0pt, outer sep = 0pt,minimum size = 20 pt,draw}}

  \def\x{1.5+1}
  \def\y{-4.0}
  \Vertex[x=\x-2.0, y=\y-1.0,L={(u_1,\tilde{x})}]{u_1s_1}
  \Vertex[x=\x+0.0, y=\y-0.0,L={(u_2,\tilde{x})}]{u_2s_1}
  \Vertex[x=\x+0.0, y=\y-2.0,L={(u_5,\tilde{x})}]{u_5s_1}
  \Edge[label = a](u_1s_1)(u_2s_1)
  \Edge[label = b](u_1s_1)(u_5s_1)
  \def\x{5.5+1}
  \def\y{-3.0}
   \Vertex[x=\x-1.0, y=\y-3.0,L={(\tilde{y},u_6)}]{t_1u_6}

\def\x{5.5+1}
  \def\y{-1.0}
 \Vertex[x=\x-1.0, y=\y-3.0,L={(\tilde{y},u_3)}]{t_1u_3}

  \def\x{7.5+1}
  \def\y{-2.0}
  \Vertex[x=\x-1.0, y=\y-3.0,L={(\tilde{y},u_4)}]{t_1u_4}
  \Edge[label = s](u_2s_1)(t_1u_6)
  \Edge[label = s](u_5s_1)(t_1u_6)
  \Edge[label = s](u_5s_1)(t_1u_3)
  \Edge[label = s](u_2s_1)(t_1u_3)
  \Edge[label = d](t_1u_6)(t_1u_4)
  \Edge[label = c](t_1u_3)(t_1u_4)
  \Edge(u_2s_1)(t_1u_6)
  \Edge(u_5s_1)(t_1u_6)
  \Edge(u_5s_1)(t_1u_3)
  \Edge(u_2s_1)(t_1u_3)

  \def\x{-6.0}
  \def\y{-0.0}
\draw[circle, -,dashed, very thick,rounded corners=8pt] (\x-0.5,\y+0.5)--(\x-0.5,\y+1.7) --(\x+2.5,\y+1.7) -- (\x+2.5,\y-1.7) -- (\x-0.5,\y-1.7) --  (\x-0.5,\y+0.5);
  \def\x{-6.1}
  \def\y{-0.25}
\draw[circle, -,dotted, very thick,rounded corners=8pt] (\x+1.6,\y+0.5)--(\x+1.6,\y+1.7) --(\x+2.5,\y+1.7) -- (\x+2.5,\y-1.2) -- (\x+1.6,\y-1.2) --  (\x+1.6,\y+0.5);
  \def\x{-2.0}
  \def\y{-0.0}
\draw[circle, -,dashed, very thick,rounded corners=8pt] (\x-0.5,\y+0.5)--(\x-0.5,\y+1.7) --(\x+2.5,\y+1.7) -- (\x+2.5,\y-1.7) -- (\x-0.5,\y-1.7) --  (\x-0.5,\y+0.5);
  \def\x{-4.0}
  \def\y{-0.25}
\draw[circle, -,dotted, very thick,rounded corners=8pt] (\x+1.6,\y+0.5)--(\x+1.6,\y+1.7) --(\x+2.5,\y+1.7) -- (\x+2.5,\y-1.2) -- (\x+1.6,\y-1.2) --  (\x+1.6,\y+0.5);
  \def\x{-1.0+1}
  \def\y{-5.0}
\draw[circle, -,dashed, very thick,rounded corners=8pt] (\x-0.5,\y+0.5)--(\x-0.5,\y+2.4) --(\x+8.5,\y+2.4) -- (\x+8.5,\y-2.2) -- (\x-0.5,\y-2.2) --  (\x-0.5,\y+0.5);
  \def\x{-1.0+1}
  \def\y{-5.0}
\draw[circle, -,dotted, very thick,rounded corners=8pt] (\x+1.6,\y+0.5)--(\x+1.6,\y+1.6) --(\x+6.4,\y+1.6) -- (\x+6.4,\y-2.0) -- (\x+1.6,\y-2.0) --  (\x+1.6,\y+0.5);

\end{tikzpicture}
}
\end{center}
\caption{Decomposition of $G$ into $G/Y$ and $G/X$, where the arcs of $[X,Y]$ arc-induce a complete bipartite subgraph $B(Z_1,Z_2)$ of $G$ with arcs with the same label pair. The dashed regions indicate the vertex sets $X$, $Y$ and $V(G/Y\boxbackslash G/X)$. The dotted regions indicate the vertex sets $Z_1$, $Z_2$ and $V(B(Z_1,Z_2)/Z_2\boxbackslash B(Z_1,Z_2)/Z_1$.}
  \label{FirstCounterDecomposition}
\end{figure}
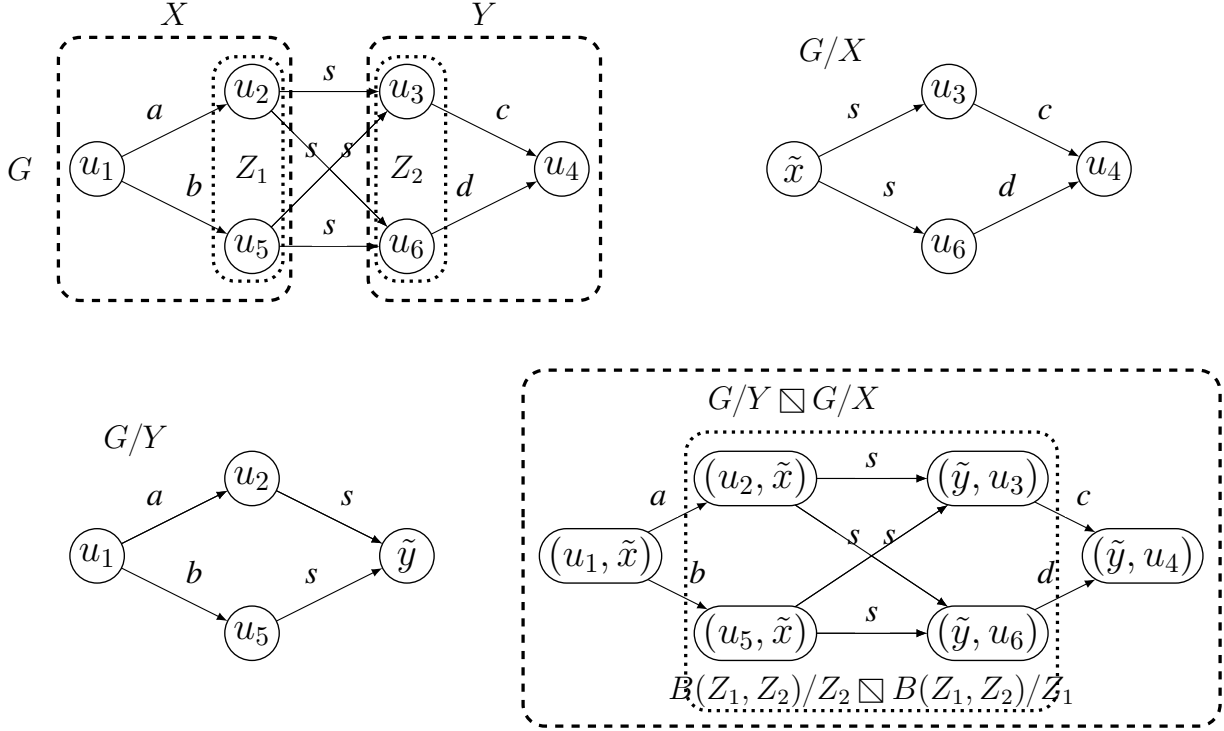
The only difference between Theorem~\ref{theorem_1} and Theorem~\ref{theorem_3} is that the arcs of $[X,Y]$ must have unique label pairs in Theorem~\ref{theorem_1}, whereas this is not required in Theorem~\ref{theorem_3}.
To relax this requirement of Theorem~\ref{theorem_1}, we require that any set of all arcs of $[X,Y]$ with identical label pairs must arc-induce a complete bipartite graph. 
By Lemma~\ref{lemma1}, these complete bipartite graphs are decomposable.
Then we have that all arcs of a complete bipartite subgraph $B(X_1,Y_1), X_1\subseteq X, Y_1\subseteq Y,$ of $G$ with the same label pair are synchronising arcs.
Furthermore, all other arcs of $G$ have label pairs different from the label pairs of $B(X_1,Y_1)$.
Therefore, Lemma~\ref{lemma1} together with Theorem~\ref{theorem_1} gives
$G\cong G/Y\boxbackslash G/X$, which we prove in Theorem~\ref{theorem_3}.
\begin{theorem}\label{theorem_3}
Let $G$ be a graph, let $X$ be a nonempty proper subset of $V(G)$, and let $Y=V(G)\setminus X$. 
Suppose that the graph $G\{[X,Y]\}$ is a clean bipartite subgraph of $G$ and that the arcs of $G/X$ and $G/Y$ corresponding to the arcs of $[X,Y]$ are the only synchronising arcs of $G/X$ and $G/Y$. 
If $S'(G)\subseteq X$ and $[X,Y]$ has no backward arcs, then $G\cong G/Y\boxbackslash G/X$.
\end{theorem}

\begin{proof}
It clearly suffices to define a mapping $\phi: V(G)\rightarrow V(G/Y\boxbackslash G/X)$ and to prove that $\phi$ is an isomorphism from $G$ to $G/Y\boxbackslash G/X$.

Let $\tilde{x}$ and $\tilde{y}$ be the new vertices replacing the sets $X$ and $Y$ when defining $G/X$ and $G/Y$, respectively. 
Consider the mapping $\phi: V(G)\rightarrow V(G/Y\boxbackslash G/X)$ defined by $\phi(u)=(u,\tilde{x})$ for all $u\in X$ and $\phi(v)=(\tilde{y},v)$ for all $v\in Y$.  
Then $\phi$ is obviously a bijection if  $V(G/Y\boxbackslash G/X)=Z$, where $Z$ is defined as $Z=\{(u,\tilde{x})\mid u\in X\}\cup \{(\tilde{y},v) \mid v\in Y\}$. We are going to show this later by arguing that all the other vertices of $G/Y\,\Box\, G/X$ will disappear from $G/Y\boxtimes G/X$. But first we are going to prove the following claim. 

\begin{claim}\label{claim3}
The subgraph of $G/Y\boxtimes G/X$ induced by $Z$ is isomorphic to $G$.
\end{claim}

\begin{proof}
Obviously, $\phi$ is a bijection from $V(G)$ to $Z$. It remains to show that this bijection preserves the arcs and their label pairs. 
By the definition of the Cartesian product, for each arc $a\in A(G)$ with $\mu(a)=(u,v)$ for $u\in X$ and $v\in X$, there exists an arc $b$ in $G/Y\boxtimes G/X$ with $\mu(b)=((u,\tilde{x}),(v,\tilde{x}))=(\phi(u),\phi(v))$ and $\lambda(b)=\lambda(a)$. 
This is because the arc $a\notin [X,Y]$, and hence $a$ is not a synchronising arc of $G/Y$ with respect to $G/X$ (by hypothesis).
Likewise, for each arc $a\in A(G)$ with $\mu(a)=(u,v)$ for $u\in Y$ and $v\in Y$, there exists an arc $b$ in $G/Y\boxtimes G/X$ with $\mu(b)=((\tilde{y},u),(\tilde{y},v))=(\phi(u),\phi(v))$ and $\lambda(b)=\lambda(a)$. 

Next, each arc $a\in A(G)$ with $\mu(a)=(u,v), u\in X$ and $v\in Y$, is an arc of $[X,Y]$.
Furthermore, all arcs in $[X,Y]$ with the same label pair arc-induce a clean  bipartite subgraph of $G$ (by hypothesis).
Then, by Lemma~\ref{lemma1}, for each arc $a\in [X,Y]$ with $\mu(a)=(u,v)$ there exists an arc $b$ with $\mu(b)=((u,\tilde{x}),(\tilde{y},v))=(\phi(u),\phi(v))$ and $\lambda(b)=\lambda(a)$.
Because the arcs of $[X,Y]$ are the only synchronising arcs we have the arc set $\{(u,\tilde{x})(\tilde{y},v)\mid u\in X, v\in Y\}$ in $G/X\boxtimes G/Y$.
Concluding, for each arc $a\in A(G)$ with $\mu(a)=(u,v), u,v\in V(G),$ there is an arc $b$ with $\mu(b)=((u,\tilde{x}),(\tilde{y},v))=(\phi(u),\phi(v)), (u,\tilde{x}),(\tilde{y},v)\in V(G/Y\boxtimes G/X)$ and $\lambda(b)=\lambda(a)$.
Hence, the subgraph of $G/Y\boxtimes G/X$ induced by $Z$ is isomorphic to $G$.
This completes the proof of Claim~\ref{claim3}. 
\end{proof}

We continue with the proof of Theorem~\ref{theorem_3}. It remains to show that all other vertices of $G/Y\,\Box\, G/X$, except for the vertices of $Z$, disappear from $G/Y\boxtimes G/X$. This is clear for the vertex $(\tilde{y},\tilde{x})$: all the arcs of $G/Y\,\Box\, G/X$ corresponding to the arcs of $[X,Y]$ are synchronising arcs of $G/Y$ and $G/X$, so they disappear from  $G/Y\boxtimes G/X$. Hence, $(\tilde{y},\tilde{x})$ has in-degree 0 (and out-degree 0) in $G/Y\boxtimes G/X$, while it has $level > 0$ in $G/Y\,\Box\, G/X$. For the other vertices, the argument is as follows.

The vertex set of $G/Y\,\Box\, G/X$ consists of $Z\cup\{(\tilde{y},\tilde{x})\}$ and the vertex set $X\times Y$. 
We will argue that all vertices of $X\times Y$ will eventually disappear from $G/Y\boxtimes G/X$.

Therefore, we claim that all $(u,v)\in X\times Y$ have $level >0$ in  $G/Y\,\Box\, G/X$. This is obvious if $u$ has $level >0$ in $G[X]$  or $v$ has  $level >0$ in  $G[Y]$.
Now, let $(u,v)\in X\times Y$ such that $u$ has $level\, 0$ in $G[X]$ and $v$ has  $level\, 0$ in  $G[Y]$.
Then the claim follows from the fact that $v$ has at least one in-arc from a vertex in $X$, since $S'(G)\subseteq X$.
Furthermore, since $v$ has only in-arcs from vertices in $X$ and $u$ has no in-arcs at all, $(u,v)$ has $level\, 0$ in  $G/Y\boxtimes G/X$.
This is because  all arcs $(u,v)\in A(G)$ are in $[X,Y]$, hence they correspond to synchronising arcs in $G/Y$ with respect to $G/X$.
Concluding, all vertices $(u,v)\in X\times Y$ such that $u$ has $level\, 0$ in $G[X]$ and $v$ has  $level\, 0$ in  $G[Y]$ disappear from $G/Y\boxtimes G/X$, together with all the arcs with tail $(u,v)$ for all such vertices $(u,v)\in X\times Y$.
If after this first step there are still vertices of $X\times Y$ left in $G/Y\boxtimes G/X$, we can repeat the above arguments step by step for such remaining vertices $(u,v)\in X\times Y$ for which $(u,v)$ has the lowest level in what has remained from $G/Y\boxtimes G/X$. Since $G/Y\boxtimes G/X$ is acyclic, it is clear that all vertices of  $X\times Y$ disappear one by one from $G/Y\boxtimes G/X$.  
This completes the proof of Theorem~\ref{theorem_3}. 
\end{proof}

We continue with the proof of Theorem~\ref{theorem_4} which relaxes the requirement of Theorem~\ref{theorem_2} that all the arcs of $[X_1,Y]$ have distinct label pairs, all the arcs of $[Y,X_2]$ have distinct label pairs and all the arcs of $[X_1,X_2]$ have distinct label pairs.
In Figure~\ref{Example1}, where the graph $G$ contains a non-trivial complete bipartite subgraph for which all arcs have identical label pairs, we have shown in a simple example how the graph $G$ can be decomposed into the graphs $G/Y$ and $G/X_1/X_2$ such that $G\cong G/Y\boxbackslash G/X_1/X_2$.
In Theorem~\ref{theorem_4}, we use the proof of Theorem~\ref{theorem_2} given in \cite{dam} and modify this proof to support complete bipartite subgraphs of $G$ with arcs in $[X_1,Y]$ with the same label pair and arcs in $[Y,X_2]$ with the same label pair and (not necessarily complete) bipartite subgraphs of $G$ with arcs in $[X_1,X_2]$ with the same label pair.

\begin{theorem}\label{theorem_4}
Let $G$ be a graph, and let $X_1$, $X_2$ and $Y=V(G)\setminus (X_1\cup X_2)$ be three disjoint nonempty subsets of $V(G)$. 
Suppose that the graph $G\{[X_1,Y]\}$ is a clean bipartite subgraph of $G$, the graph $G\{[Y,X_2]\}$ is a clean bipartite subgraph of $G$, the arcs of $[X_1,X_2]$ have no label pairs in common with any arc in $[X_1,Y]\cup[Y,X_2]$, and the arcs of $G/X_1/X_2$ and $G/Y$ corresponding to the arcs of $[X_1,Y]\cup [Y,X_2]\cup [X_1,X_2]$ are the only synchronising arcs of $G/X_1/X_2$ and $G/Y$. 
If $S'(G)\subseteq X_1$, and $[X_1,Y]$, $[Y,X_2]$ and $[X_1,X_2]$ have no backward arcs, then $G\cong G/Y\boxbackslash G/X_1/X_2$.
\end{theorem}
\begin{proof}
It suffices to define a mapping $\phi: V(G)\rightarrow V(G/Y\boxbackslash G/X_1/X_2)$ and to prove that $\phi$ is an isomorphism from $G$ to $G/Y\boxbackslash G/X_1/X_2$.

Let $\tilde{x}_1$, $\tilde{x}_2$ and $\tilde{y}$ be the new vertices replacing the sets $X_1$, $X_2$ and $Y$ when defining $G/X_1/X_2$ and $G/Y$, respectively. 
Consider the mapping $\phi: V(G)\rightarrow V(G/Y\boxbackslash G/X_1/X_2)$ defined by $\phi(u)=(u,\tilde{x}_1)$ for all $u\in X_1$, $\phi(v)=(v,\tilde{x}_2)$  for all $v\in X_2$ and $\phi(w)=(\tilde{y},w)$ for all $w\in Y$. 

\noindent 
Then $\phi$ is clearly a bijection if  $V(G/Y\boxbackslash G/X_1/X_2)=Z$, where $Z$ is defined as $Z=\{(u,\tilde{x}_1)\mid u\in X_1\}\cup\{(v,\tilde{x}_2)\mid v\in X_2\}\cup \{(\tilde{y},w) \mid w\in Y\}$. We are going to show this later by arguing that all the other vertices of $G/Y\,\Box\, G/X_1/X_2$ will disappear from $G/Y\boxtimes G/X_1/X_2$. But first we are going to prove the following claim. 

\begin{claim}\label{claim4}
The subgraph of $G/Y\boxtimes G/X_1/X_2$ induced by $Z$ is isomorphic to $G$.
\end{claim}

\begin{proof}
Obviously, $\phi$ is a bijection from $V(G)$ to $Z$. It remains to show that this bijection preserves the arcs and their label pairs. By the definition of the Cartesian product, for each arc $a\in A(G)$ with $\mu(a)=(u,v)$ for $u\in X_1$ and $v\in X_1$, there exists an arc $b$ in $G/Y\boxtimes G/X_1/X_2$ with $\mu(b)=((u,\tilde{x}_1),(v,\tilde{x}_1))=(\phi(u),\phi(v))$ and $\lambda(b)=\lambda(a)$. Likewise, for each arc $a\in A(G)$ with $\mu(a)=(u,v)$ for $u\in Y$ and $v\in Y$, there exists an arc $b$ in $G/Y\boxtimes G/X_1/X_2$ with $\mu(b)=((\tilde{y},u),(\tilde{y},v))=(\phi(u),\phi(v))$ and $\lambda(b)=\lambda(a)$, and for each arc $a\in A(G)$ with $\mu(a)=(u,v)$ for $u\in X_2$ and $v\in X_2$, there exists an arc $b$ in $G/Y\boxtimes G/X_1/X_2$ with $\mu(b)=((u,\tilde{x}_2),(v,\tilde{x}_2))=(\phi(u),\phi(v))$ and $\lambda(b)=\lambda(a)$.
Next, we distinguish two cases, the arcs of $[X_1,Y]$ and $[Y,X_2]$, and the arcs of $[X_1,X_2]$.

Firstly, consider the arcs of $[X_1,Y]$ and $[Y,X_2]$.
By hypothesis, the arcs with identical label pairs of $[X_1,Y]$ arc-induce a complete bipartite subgraph $B(Z_1,Z_2)$ of $G$ and the arcs with identical label pairs of $[Y,X_2]$ arc-induce a complete bipartite subgraph $B(Z_3,Z_4)$ of $G$.
Let $Z_1=\{u_1,\ldots, u_m\}\subseteq X_1$ and $Z_2=\{v_1,\ldots, v_n\}\subseteq Y$ and let $Z_3=\{u'_{1},\ldots, u'_{m'}\}\subseteq Y$ and $Z_4=\{v'_1,\ldots, v'_{n'}\}\subseteq X_2$. 
Let all arcs of $B(Z_1,Z_2)$ and $B(Z_3,Z_4)$ have the same label pair $\alpha$.

According to Lemma~\ref{lemma1}, $B(Z_1,Z_2)$ can be decomposed in $B(Z_1,Z_2)/Y$ and $B(Z_1,Z_2)/X_1$ with $B(Z_1,Z_2)\cong B(Z_1,Z_2)/Y\boxbackslash B(Z_1,Z_2)/X_1$ and $B(Z_3,Z_4)$ can be decomposed in $B(Z_3,Z_4)$ $/Y$ and $B(Z_3,Z_4)/X_2$ with $B(Z_3,Z_4)\cong B(Z_3,Z_4)/Y\boxbackslash B(Z_3,Z_4)/X_2$.
Note that $B(Z_1,Z_2)/X_1$ $=B(Z_1,Z_2)/X_1/$ $X_2$ because $V(B(Z_1,Z_2)/X_1)\cap X_2=\emptyset$ and $B(Z_3,Z_4)/X_2=B(Z_3,Z_4)/X_1/$ $X_2$ because $V(B(Z_3,Z_4)/X_2)\cap X_1=\emptyset$.
Furthermore, note that $B(Z_1,Z_2)$ and $B(Z_3,Z_4)$ do not have backward arcs.
For $B(Z_1,Z_2)/Y$ and $B(Z_1,Z_2)/X_1$, we have the arc sets $A(B(Z_1,Z_2)/Y)=\{a_i\mid \mu(a_i)=(u_i,\tilde{y}),i=1,\ldots,m\}$ and $A(B(Z_1,Z_2)/X_1)=\{b_j\mid \mu(b_j)=(\tilde{x}_1,v_j),j=1,$ $\ldots, n\}$, respectively and for $B(Z_3,Z_4)/Y$ and $B(Z_3,Z_4)$ $/X_2$, we have the arc sets $A(B(Z_3,Z_4)$ $/Y)=\{c_i\mid \mu(c_i)=(\tilde{y}, v'_i),i=1,\ldots, n'\}$ and $A(B(Z_3,Z_4)$ $/X_2)=\{d_j\mid \mu(d_j)=(u'_j,\tilde{x}_2),j=1,\ldots, m'\}$, respectively.
Because these arcs are the only arcs synchronising over label pair $\alpha$, we have the arc set $\{e_{i,j}\mid \mu(e_{i,j})=((u_i,\tilde{x}_1),(\tilde{y},v_j)), i=1,\ldots,m,j=1,\ldots,n\}\cup \{f_{j',j}\mid \mu(f_{j',j})=((\tilde{y},\tilde{x}_1),(v'_{j'},v_j)), j=1,\ldots,n$, $j'=1,\ldots,n'\}\cup \{g_{i,i'}\mid \mu(g_{i,i'})=((u_i,u'_{i'}),(\tilde{y},\tilde{x}_2)), i=1,\ldots,m,$ $i'=1,\ldots,m'\}\cup \{h_{i',j'}\mid \mu(h_{i',j'})=((\tilde{y},u'_{i'}),(v'_{j'},\tilde{x}_2)), i'=1,\ldots,m',j'=1,\ldots,n'\} $ in $G/Y\boxtimes G/X$.
Therefore, for each arc $a\in A(G)$ with $\mu(a)=(u_i,v_j)$ for $u_i\in Z_1$ and $v_j\in Z_2$, there exists an arc $b\in G/Y\boxtimes G/X_1/X_2$ with $\mu(b)=((u_i,\tilde{x}_1)(\tilde{y},v_j))=(\phi(u_i),\phi(v_j))$ and $\lambda(b)=\lambda(a)$, and for each arc $c\in A(G)$ with $\mu(c)=(u'_{i'},v'_{j'})$ for $u'_{i'}\in Z_3$ and $v'_{j'}\in Z_4$, there exists an arc $d\in G/Y\boxtimes G/X_1/X_2$ with $\mu(d)=((\tilde{y},u'_{i'})(v'_{j'},\tilde{x}_2))=(\phi(u'_{i'}),\phi(v'_{j'}))$ and $\lambda(c)=\lambda(a)$.

It is sufficient to prove the preservation of the arcs with the same label pair for $B(Z_1,Z_2)$ and $B(Z_3,Z_4)$.
If $B(Z_3,Z_4)$ does not exist, we do not have the subgraphs $B(Z_1,Z_2)/Y\boxbackslash B(Z_3,Z_4)$ $/X_2, B(Z_3,Z_4)/Y\boxbackslash B(Z_1,Z_2)/X_1$ and $B(Z_3,Z_4)/Y\boxbackslash B(Z_3,Z_4)/X_2$ of $G/Y\boxbackslash G/X_1/X_2$ and if $B(Z_1,Z_2)$ does not exist, we do not have the subgraphs $B(Z_1,Z_2)/Y\boxbackslash B(Z_3,Z_4)/X_2$, $B(Z_3,Z_4)$ $/Y\boxbackslash B(Z_1,Z_2)/X_1$ and $B(Z_1,Z_2)/Y\boxbackslash B(Z_1,Z_2)/X_1$ of $G/Y\boxbackslash G/X_1/X_2$.
Therefore, this observation reduces the proof for $B(Z_1,Z_2)$ and $B(Z_3,Z_4)$ with arcs with identical label pairs to the proof for $B(Z_1,Z_2)$ with arcs with identical label pairs and the proof for $B(Z_3,Z_4)$ with arcs with identical label pairs.

Secondly, let $Z_1\subseteq X_1$ and $Z_2\subseteq X_2$.
Let $B(Z_1,Z_2)$ be a bipartite subgraph of $G$ with vertex sets $Z_1=\{u_1,\ldots, u_m\}\subseteq X_1$ and $Z_2=\{v_1,\ldots, v_n\}\subseteq X_2$ where each arc $a\in A(B(Z_1,Z_2))$ has the same label pair $\alpha$.
Then the contraction $G/Y$ will leave all arcs $a$ of $B(Z_1,Z_2)$ with $\mu(a)=(u_i,v_j)$ and $\lambda(a)=\alpha$  unchanged, therefore these arcs $a$ correspond to arcs $b$ of $B(Z_1,Z_2)/Y$ with $\mu(b)=(u_i,v_j)$ and $\lambda(b)=\lambda(a)$. 
The contraction $G/X_1/X_2$ will replace all vertices $u_i$ of $X_1$ by one vertex $\tilde{x}_1$ and all vertices $v_j$ of $X_2$ by one vertex $\tilde{x}_2$, and therefore, all the arcs $a$ of $B(Z_1,Z_2)$ with $\mu(a)=(u_i,v_j)$ and $\lambda(a)=\alpha$ are replaced by one arc $c$ with $\mu(c)=(\tilde{x}_1,\tilde{x}_2)$ and $\lambda(c)=\lambda(a)$ of $G/X_1/X_2$.
Because all arcs $b$ of $B(Z_1,Z_2)\subseteq G/Y$ are synchronous arcs with respect to the arc $c$ of $G/X_1/X_2$, we have that each pair of arcs $b$ and $c$ correspond with an arc $d$ of $B(Z_1,Z_2)/Y\boxbackslash B(Z_1,Z_2)/X_1/X_2$ with $\mu(d)=((u_i,\tilde{x}_1),(v_j,\tilde{x}_2))$ and $\lambda(d)=\lambda(a)$. 
Since there are no backward arcs in $[X_1,Y]$, $[Y,X_2]$ and $[X_1,X_2]$, the above arcs are the only arcs in $G/Y\boxtimes G/X_1/X_2$ induced by the vertices of $Z$.
The proof in case of $B(Z_3,Z_4)$ is similar.
This completes the proof of Claim~\ref{claim4}. 
\end{proof}
We continue with the proof of Theorem~\ref{theorem_4}. It remains to show that all other vertices of $G/Y\boxtimes G/X_1/X_2$, except for the vertices of $Z$, disappear from $G/Y\boxtimes G/X_1/X_2$. 
This is clear for the vertex $(\tilde{y},\tilde{x}_1)$: all the arcs of $G/Y\,\Box\, G/X_1/X_2$ corresponding to the arcs of $[X_1,Y]$ are synchronising arcs of $G/Y$ and $G/X_1/X_2$, so they disappear from  $G/Y\boxtimes G/X_1/X_2$. 
Hence, $(\tilde{y},\tilde{x}_1)$ has in-degree 0 in $G/Y\boxtimes G/X_1/X_2$, while it has $level > 0$ in $G/Y\,\Box\, G/X_1/X_2$. 
For the other vertices, the argument is as follows.

The vertex set of $G/Y\,\Box\, G/X_1/X_2$ consists of the union of $Z\cup\{(\tilde{y},\tilde{x}_1),(\tilde{y},\tilde{x}_2)\}$ and the vertex sets $(X_1\cup X_2)\times Y$,  $X_1\times \{\tilde{x}_2\}$ and  $X_2\times \{\tilde{x}_1\}$. 
We will argue that all vertices of $(X_1\cup X_2)\times Y$,  $X_1\times \{\tilde{x}_2\}$ and $X_2\times \{\tilde{x}_1\}$, as well as the vertex  $(\tilde{y},\tilde{x}_2)$ will eventually disappear from $G/Y\boxtimes G/X_1/X_2$.

Firstly, we claim that all $(u,v)\in X_1\times Y$ have $level >0$ in  $G/Y\,\Box\, G/X_1/X_2$. This is obvious if $u$ has $level >0$ in $G[X_1]$  or $v$ has  $level >0$ in  $G[Y]$.
Now let $(u,v)\in X_1\times Y$ such that $u$ has $level\, 0$ in $G[X_1]$ and $v$ has  $level\, 0$ in  $G[Y]$.
Then the claim follows from the fact that $v$ has at least one in-arc from a vertex in $X_1$, since $S'(G)\subseteq X_1$.
Furthermore, since $v$ has only in-arcs from vertices in $X_1$ and $u$ has no in-arcs at all, $(u,v)$ has $level\, 0$ in  $G/Y\boxtimes G/X_1/X_2$.
Hence, all vertices $(u,v)\in X_1\times Y$ such that $u$ has $level\, 0$ in $G[X_1]$ and $v$ has  $level\, 0$ in  $G[Y]$ disappear from $G/Y\boxtimes G/X_1/X_2$, together with all the arcs with tail $(u,v)$ for all such vertices $(u,v)\in X_1\times Y$.
If after this first step there are still vertices of $X_1\times Y$ left in $G/Y\boxtimes G/X_1/X_2$, we can repeat the above arguments step by step for such remaining vertices $(u,v)\in X_1\times Y$ for which $(u,v)$ has the lowest level in what has remained from $G/Y\boxtimes G/X_1/X_2$. Since $G/Y\boxtimes G/X_1/X_2$ is acyclic, it is clear that all vertices of  $X_1\times Y$ disappear one by one from $G/Y\boxtimes G/X_1/X_2$. 
Now, since $(\tilde{y},\tilde{x}_2)$ has possibly only in-arcs from vertices $(u,v)\in X_1\times Y$, $(\tilde{y},\tilde{x}_2)$ will disappear as well.

Next, we claim that all $(u,v)\in X_2\times Y$ have $level >0$ in  $G/Y\,\Box\, G/X_1/X_2$. This is obvious if $u$ has $level >0$ in $G[X_2]$  or $v$ has  $level >0$ in  $G[Y]$.
Now let $(u,v)\in X_2\times Y$ such that $u$ has $level\, 0$ in $G[X_2]$ and $v$ has  $level\, 0$ in  $G[Y]$.
Then the claim follows from the fact that $u$ has at least one in-arc from a vertex in $Y$, since $[Y,X_2]$ has only forward arcs.
Furthermore, since $u$ has only in-arcs from vertices in $Y$ and $v$ has no in-arcs at all, $(u,v)$ has $level\, 0$ in  $G/Y\boxtimes G/X_1/X_2$.
Hence, all vertices $(u,v)\in X_2\times Y$ such that $u$ has $level\, 0$ in $G[X_2]$ and $v$ has  $level\, 0$ in  $G[Y]$ disappear from $G/Y\boxtimes G/X_1/X_2$, together with all the arcs with tail $(u,v)$ for all such vertices $(u,v)\in X_2\times Y$.
If after this first step there are still vertices of $X_2\times Y$ left in $G/Y\boxtimes G/X_1/X_2$, we can repeat the above arguments step by step for such remaining vertices $(u,v)\in X_2\times Y$ for which $(u,v)$ has the lowest level in what has remained from $G/Y\boxtimes G/X_1/X_2$. Since $G/Y\boxtimes G/X_1/X_2$ is acyclic, it is clear that all vertices of  $X_2\times Y$ disappear one by one from $G/Y\boxtimes G/X_1/X_2$. 

We continue with the claim that all $(u,\tilde{x}_1)\in X_2\times \{\tilde{x}_1\}$ have $level >0$ in  $G/Y\,\Box\, G/X_1/X_2$. This is obvious if $u$ has $level >0$ in $G[X_2]$.
Now let $(u,\tilde{x}_1)\in X_2\times \{\tilde{x}_1\}$ such that $u$ has $level\, 0$ in $G[X_2]$.
Then the claim follows from the fact that $u$ has at least one in-arc from a vertex in $Y$, since $[Y,X_2]$ has only forward arcs.
Furthermore, since $u$ has only in-arcs from vertices in $Y$ and $\tilde{x}_1$ has no in-arcs at all, $(u,\tilde{x}_1)$ has $level\, 0$ in  $G/Y\boxtimes G/X_1/X_2$.
Hence, all vertices $(u,\tilde{x}_1)\in X_2\times \{\tilde{x}_1\}$ such that $u$ has $level\, 0$ in $G[X_2]$  disappear from $G/Y\boxtimes G/X_1/X_2$, together with all the arcs with tail $(u,\tilde{x}_1)$ for all such vertices $(u,\tilde{x}_1)\in X_2\times \{\tilde{x}_1\}$.
If after this first step there are still vertices of $X_2\times \{\tilde{x}_1\}$ left in $G/Y\boxtimes G/X_1/X_2$, we can repeat the above arguments step by step for such remaining vertices $(u,\tilde{x}_1)\in X_2\times \{\tilde{x}_1\}$ for which $(u,\tilde{x}_1)$ has the lowest level in what has remained from $G/Y\boxtimes G/X_1/X_2$. Since $G/Y\boxtimes G/X_1/X_2$ is acyclic, it is clear that all vertices of  $X_2\times \{\tilde{x}_1\}$ disappear one by one from $G/Y\boxtimes G/X_1/X_2$. 

Finally, we claim that all $(u,\tilde{x}_2)\in X_1\times \{\tilde{x}_2\}$ have $level >0$ in  $G/Y\,\Box\, G/X_1/X_2$. This is obvious if $u$ has $level >0$ in $G[X_1]$. 
Now let $(u,\tilde{x}_2)\in X_1\times \{\tilde{x}_2\}$ such that $u$ has $level\, 0$ in $G[X_1]$.
Then the claim follows from the fact that $\tilde{x}_2$ has at least one in-arc from a vertex in $Y$, since $[Y,X_2]$ has only forward arcs and $S'(G) \subseteq X_1$ by hypothesis.
Noting that $\tilde{x}_2$ has only in-arcs from vertices in $Y$, and all $u\in S'(G)\subseteq X_1$ have no in-arcs at all, clearly for all $u\in S'(G)\subseteq X_1$, $(u,\tilde{x}_2)$ has $level\, 0$ in  $G/Y\boxtimes G/X_1/X_2$.
Hence, all vertices $(u,\tilde{x}_2)\in X_1\times \{\tilde{x}_2\}$ such that $u$ has $level\, 0$ in $G[X_1]$ disappear from $G/Y\boxtimes G/X_1/X_2$, together with all the arcs with tail $(u,\tilde{x}_2)$ for all such vertices $(u,\tilde{x}_2)\in X_1\times \{\tilde{x}_2\}$. 
If after this first step there are still vertices of $X_1\times \{\tilde{x}_2\}$ left in $G/Y\boxtimes G/X_1/X_2$, we can repeat the above arguments step by step for such remaining vertices $(u,\tilde{x}_2)\in X_1\times \{\tilde{x}_2\}$ for which $(u,\tilde{x}_2)$ has the lowest level in what has remained from $G/Y\boxtimes G/X_1/X_2$. Since $G/Y\boxtimes G/X_1/X_2$ is acyclic, it is clear that all vertices of  $X_1\times \{\tilde{x}_2\}$ disappear one by one from $G/Y\boxtimes G/X_1/X_2$. 
This completes the proof of Theorem~\ref{theorem_4}. 
\end{proof}

\section{Future work}
The ultimate purpose is to create a set of decomposition theorems that, when applied to an edge-labelled acyclic directed multigraphs, will result in a set of graphs that can not be decomposed anymore using the VRSP. 
As an example, a graph $G$ that has the property that $G\cong G_1\,\Box\, G_2\cong G_1\boxbackslash G_2$ for two subgraphs $G_1,G_2$ of $G$ can not be decomposed by the theorems we have presented so far. 
Therefore, in future contributions, we will present theorems by which we can decompose graphs that contain subgraphs that have this \enquote{Cartesian} characteristic.

\end{document}